\documentclass{amsart}

\usepackage[table,xcdraw]{xcolor}
\usepackage{amssymb}
\usepackage{amsmath}
\usepackage{amsfonts}
\usepackage{amsxtra}
\usepackage{amsthm}
\usepackage{tikz}
\usepackage{tikz-cd}
\usepackage{float}
\usepackage{booktabs}
\usepackage{color}
\usepackage{xcolor}
\usepackage{graphicx}
\usepackage{enumitem}
\usepackage{venndiagram}
\usetikzlibrary{positioning,shapes,fit,arrows}
\usepackage{multirow}
\usepackage{multicol}

\newtheorem{theorem}{Theorem}[section]
\newtheorem{proposition}[theorem]{Proposition}
\newtheorem{corollary}[theorem]{Corollary}
\newtheorem{lemma}[theorem]{Lemma}
\newtheorem{exmp}[theorem]{Example}
\newtheorem{definition}[theorem]{Definition}
\newtheorem{remark}[theorem]{Remark}

\newcommand{\abf}{\mathbf{a}}
\newcommand{\I}{\mathrm{I}}
\newcommand{\M}{\mathrm{M}}
\newcommand{\PM}{\mathrm{PM}}
\newcommand{\FI}{\mathrm{FI}}
\newcommand{\Id}{\mathrm{Id}}
\newcommand{\Ker}{\mathrm{Ker}}
\newcommand{\Rad}{\mathrm{Rad}}
\newcommand{\Md}{\mathbb{M}_d}
\newcommand{\HomM}{\mathrm{Hom}(M,N)}

\def\n{\mathbb{N}}
\def\z{\mathbb{Z}}
\def\q{\mathbb{Q}}
\def\r{\mathbb{R}}

\newcommand{\NVL}{\mathsf{NVL}}
\newcommand{\AVL}{\mathsf{AVL}}
\newcommand{\CIL}{\mathsf{CIL}}
\def\inv{^{-1}}

\begin{document}

\title[]{Strolling through common meadows}

\author[]{João Dias}
\author[]{Bruno Dinis}

\address[Bruno Dinis]{Departamento de Matemática, Universidade de Évora}
\email{bruno.dinis@uevora.pt}

\address[João Dias]{Departamento de Matemática, Universidade de Évora}
\email{joao.miguel.dias@uevora.pt}

\subjclass[2010]{16U90, 06B15,16B50}

\keywords{Common meadows, unital commutative rings, directed lattices, category of meadows}

\begin{abstract}
    This paper establishes a connection between rings, lattices and common meadows.
    Meadows are commutative and associative algebraic structures with two operations (addition and multiplication) with additive and multiplicative identities and for which inverses are total. Common meadows are meadows that introduce, as the inverse of zero, an error term $\abf$ which is absorbent for addition.
    
    We show that common meadows are unions of rings which are ordered by a partial order that defines a lattice. These results allow us to extend some classical algebraic constructions to the setting of common meadows.  We also briefly consider common meadows from a categorical perspective.
\end{abstract}
\maketitle
\section{Introduction}

In classical mathematics it doesn't make sense to divide by zero, as some false statements can be immediately derived. There are countless "wrong proofs" in which the wrong part is precisely a (more or less conspicuous) division by zero. One way to overcome this limitation is to give up on some usual algebraic properties, e.g.\ the cancellation law, and allowing for $0 \cdot x$ to be different from $0$. 

In this paper we use this idea, which can be implemented by the notion of common meadow, and establish a connection between common meadows and the well-known structures of rings and lattices. We believe that these connections are a first step into a fruitful line of research that starts by investigating the algebraic properties of common meadows, as we will do below. 

Meadows were introduced by Bergstra and Tucker in \cite{10.1145/1219092.1219095} as algebraic structures, given by equational theories, where it makes sense to divide by zero. To be more specific, a meadow is a sort of commutative ring with a multiplicative identity element and a total multiplicative inverse operation. Two of the main classes of meadows are involutive meadows, in which the inverse of zero is zero, and common meadows (introduced in \cite{Bergstra2015}), in which the inverse of zero is a term $\abf$ which is maximal, in the sense that $x+\abf=\abf$, for every element $x$ in the meadow. In common meadows $0 \cdot x$ doesn't have to be equal to $0$. 
  
  Even though meadows were only recently introduced, the subject is revealing to be of interest, mostly as datatypes given by equational axiomatizations (see  e.g.\ \cite{Bergstra2008,10.1145/1219092.1219095,BP(20),bergstra2020arithmetical}) allowing for simple term rewriting systems which are easier to automate in formal reasoning  \cite{bergstra2020arithmetical,bergstra2023axioms}. More recently, Bottazzi and the second author \cite{Dinis_Bottazzi} found a connection with nonstandard analysis which provides new models for both involutive and common meadows.
  Our main motivation for writing this paper is of a different nature though, as we want to study common meadows as algebraic structures on their own right.


 In Section \ref{S:PreMeadows}, we introduce the notion of \emph{pre-meadow} as an algebraic structure satisfying the properties of common meadows except the ones related with the inverse. We prove that every pre-meadow is a disjoint union of rings. As it turns out, under certain conditions, pre-meadows can be extended to common meadows in an essentially unique way. A natural order relation on meadows allows us to show, in Section \ref{S:DirectedLattices}, that every common meadow $M$ is in relation with a particular kind of lattice of rings, which we call \emph{directed lattice}, and, conversely,  that one can associate a common meadow to a certain class of directed lattices of rings over a lattice. These results allow us to extend the usual algebraic notions studied in rings to the setting of common meadows. Namely, we consider homomorphisms, ideals, kernels and isomorphisms. This is done in Section \ref{S:Algebraicconstructions}. In Section \ref{S:Alternative}, and profiting from the relations with lattices and rings unveiled in the previous sections, we consider alternative classes of common meadows, i.e.\ common meadows that satisfy some additional properties. The three 
 properties that we consider are exactly the ones considered in \cite[Section~2.3]{Bergstra2015}. We give what is essentially an alternative proof for \cite[Proposition 3.1.1]{Bergstra2015} and give a characterization for one of these conditions in terms of maximal ideals of a common meadow.
Finally, in Section \ref{S:Categories} we explore the possibility of viewing meadows through a categorical perspective.

\section{Pre-meadows}\label{S:PreMeadows}
In this section we introduce the notion of pre-meadow as a structure that satisfies the axioms of common meadows (see \cite{Bergstra2015}) not related with inverses, and show that every pre-meadow is a disjoint union of rings. Our notion of pre-meadow should not be mistaken with the notion of premeadow in \cite{bergstra2019division}.

\begin{definition}\label{D:PreMeadow}
A \emph{pre-meadow} is a structure $(P,+,-,\cdot)$ satisfying the following equations

\begin{tabular}{l c c r}
\\[-2mm]
$(\PM_1)$ &\qquad \qquad \qquad&$(x+y)+z=x+(y+z) $  & \qquad \qquad \\[2mm]
$(\PM_2)$ && $x+y=y+x $ \\[2mm]
$(\PM_3)$ && $x+0=x$  \\[2mm]
$(\PM_4)$ && $x+ (-x)=0 \cdot x$ \\[2mm]
$(\PM_5)$ && $(x \cdot y) \cdot z=x \cdot (y \cdot z)$  \\[2mm]
$(\PM_6)$ && $x \cdot y=y \cdot x $ \\[2mm]
$(\PM_7)$ && $1 \cdot x=x$ \\[2mm]
$(\PM_8)$ && $x \cdot (y+z)= x \cdot y + x \cdot z$ \\[2mm]
$(\PM_9)$ && $-(-x)=x$ \\[2mm]
$(\PM_{10})$ && $0\cdot(x+y)=0\cdot x \cdot y$ \\[2mm]
\end{tabular}
\end{definition}

A pre-meadow is then a structure such that both $(P,+)$ and $(P,\cdot)$ are monoids, linked by the distributive law ($\PM_8$) and possessing a sort of generalized zero $0\cdot x$, as given by $(\PM_4)$, for which the operations of addition and multiplication coincide, as in $(\PM_{10})$.
It is easy to see that if $R$ is a ring then it is also a pre-meadow with the property $0\cdot x = 0$, for all $x\in R$.

The following proposition covers some basic results on pre-meadows. We will use these properties throughout the paper without explicit mention. The proof is essentially the same as in \cite[Proposition~2.2.1]{Bergstra2015}, so we shall omit it.

\begin{proposition}\label{P:pre-meadowres}
Let $P$ be a pre-meadow. Then
\begin{enumerate}
    \item  $0 \cdot 0 = 0$ 
    \item $-0 = 0$
    \item $0 \cdot x = 0 \cdot (-x)$
    \item $-(x \cdot y) = x \cdot (-y)$
    \item $(-x) \cdot (-y) = x \cdot y$
    \item $(-1) \cdot x = -x $.
    \item $0\cdot (x \cdot x) = 0 \cdot x$
\end{enumerate}
\end{proposition}

In \cite{DinisBerg(11)} the notion of assembly was introduced as a sort of group with individualized zeros. The conditions were then slightly adapted in \cite{DDT} in order to show that a semigroup is a band of groups if and only if it is an assembly. We recall the latter definition of assembly below and show that the generalized zero postulated by $(\PM_4)$ allows to establish a connection between pre-meadows and assemblies. 

\begin{definition} \label{defassembley}
A nonempty semigroup $(S, \cdot)$ is called an \emph{assembly} if the following hold

\begin{enumerate}
\item[$(A_1)$] $\forall x\, \exists e = e(x)\, (xe=ex=x  \wedge \forall f\ (xf=fx=x\to ef=fe=e))$
\item[$(A_2)$]  $\forall x\, \exists s=s(x)\, (xs=sx=e(x)\wedge e(s)=e(x))$
\item [$(A_3)$] $\forall x\,\forall y\,(e(xy)=e(x)e(y))$.
\end{enumerate}

If condition $(A_3)$ is replaced by 
\begin{equation}\tag{$A_3'$}
\forall x \, \forall y \left(e(xy)=e(x)\lor e(xy)=e(y)\right)   
\end{equation}
we say that the resulting structure is a
\emph{strong assembly}.

To make explicit the functions that exist by conditions $(A_1)$ and $(A_2)$ we write $(S,\cdot,e,s)$ instead of $(S,\cdot)$. 
\end{definition}  

The functional notation $e(x)$ and $s(x)$ used above is justified by the fact that the elements $e$ and $s$  are unique (see the discussion after Definition~4.9 in \cite{DinisBerg(11)}).


\begin{proposition}\label{P:Assembly}
    Let $P$ be a pre-meadow. Define $e(x)=0\cdot x$ and $s(x)=-x$, for all $x\in P$. Then the structure $(P,+,e,s)$ is an assembly.
\end{proposition}
\begin{proof}
    For $x\in P$, we have $e(x)+x=0\cdot x + x = x,$
    and if $y+x = x$ we have 
    $$0\cdot x +y = x+(-x)+y = x + (-x) = 0\cdot x .  $$ 
    Hence $(A_1)$ is satisfied.

    For $(A_2)$, note that $$x+s(x)=x+(-x)=0\cdot x=e(x)$$ and $$e(s(x))=e(-x)=0\cdot (-x)=0\cdot x=e(x).$$

    Finally, for $(A_3)$ we have $e(x+y)=0\cdot (x+y)=0 \cdot x + 0\cdot y =e(x)+e(y)$.
\end{proof}

In \cite{DDT}, it was shown that assemblies are disjoint unions of groups. A related result can be shown for meadows. Indeed, we show that a pre-meadow is always a disjoint union of rings. 

Recall that a \emph{unital commutative ring} is a commutative ring with identity different from zero, and the \emph{zero ring} is the unique ring with only one element.
\begin{theorem}\label{T:union}
Every pre-meadow $P$ is a disjoint union of unital commutative rings or multiple copies of the zero ring, of the form $$P_z:=\{x\in P\mid 0\cdot x = z  \},$$
with zero $z$ and unit $1+z$, where $z\in 0\cdot P$. 
\end{theorem}

\begin{proof}
		Let $z\in 0\cdot P$. Clearly, both addition and multiplication in $P_z$ are associative and commutative, and multiplication is distributive with respect to addition because these properties hold in $P$. Clearly $P_z$ is closed under addition and multiplication. Let $x\in P_z$. Then 
            \begin{itemize}
                \item $x+z=1\cdot x+0\cdot x=(1+0)\cdot x = 1\cdot x=x$
                \item $x+(-x)=0\cdot x = z$
                \item $0\cdot (1+z)=0+0\cdot z = z$
                \item $x\cdot (1+z)=x+x\cdot z=x+x\cdot (0\cdot x)=x+0\cdot x\cdot x = x+0\cdot x=x$.
            \end{itemize}
            Hence, for all $z\in 0\cdot P$, the set $P_z$ is a unital commutative ring, with zero equal to $z$ and unit equal to $1+z$. Taking an element $x\in P$, then $x\in P_{0\cdot x}$, and so $P=\bigcup_{z\in 0\cdot P} P_z$. It is easy to see that the union is disjoint. 
\end{proof}

    One easily shows that $0\cdot P$ is closed under sums, since $0\cdot x+0\cdot y=0\cdot(x+y)$, by $(\PM_8)$. We recall that an element $x$ is said to be \emph{idempotent} if and only if $x\cdot x= x$. The following proposition then easily follows from \cite[Proposition 2.2.1]{Bergstra2015} and Proposition~\ref{P:pre-meadowres}.
    
    \begin{proposition}\label{P:0Dmonoid}
        Let $P$ be a pre-meadow, then $(0\cdot P,+)=(0\cdot P,\cdot)$ is a monoid where all elements are idempotent. 
    \end{proposition}

    We now define a partial order relation in pre-meadows which, as it turns out, is just the usual partial order defined in idempotent semigroups. 

    \begin{definition}\label{D:order}
  Let $P$ be a pre-meadow and $z,z'\in 0\cdot P$. We say that $z$ is \emph{less than or     equal to} $z'$, and write $z\leq z'$, if and only if $z\cdot z' = z$.  
    \end{definition}

     \begin{proposition}\label{P:order}
           The order relation $\leq$ 
           defines a semi-lattice, with maximum $0$.
        \end{proposition}
        \begin{proof}
            By Proposition \ref{P:0Dmonoid} we have that $0\cdot P$ is an idempotent commutative semigroup for the product.
            Then the order relation $\leq$ is just the usual partial order defined in idempotent semigroups.

            Let $z\in P$. Then $(0\cdot z)\cdot 0 = 0\cdot z$, which implies that $0$ is the maximum.   
        \end{proof} 

    \begin{definition}\label{D:PreMeadowA}
           We say that $M$ is a \emph{pre-meadow with $\abf$} if $M$ is a pre-meadow such that
    \begin{enumerate}
        \item There exists a unique $z\in 0\cdot M$ such that $|M_z|=1$. This element will be denoted by $\abf$
        \item For all $x\in M$ one has $x+\abf=\abf$
    \end{enumerate}
    \end{definition}

We would like to point out that our notion of pre-meadow with $\abf$ is related with the notion of \emph{weak commutative ring with $\bot$} in \cite{BergstraTucker23}.

\begin{exmp}\label{firstexmp}
    \begin{enumerate}
     \item A unital commutative ring $R$ is always a pre-meadow, however it is not a pre-meadow with $\abf$, since $0\cdot R=\{0\}$ and $$M_0:=\{x\in M\mid 0\cdot x = 0\}=R.$$
     
     \item Given a unital commutative ring $R$, the set $M=R\sqcup \{\abf\}$ is a pre-meadow with $\abf$, where for $x\in M$ we define $x +\abf=\abf$ and $x\cdot\abf=\abf$, and the product and sum of elements in $R$ coincides with the operations in $R$.
        
    \item Consider the set $M=\z\times\{0\}\cup \q\times\{1\}$. Formally, this is just the disjoint union of $\z$ and $\q$. We define the sum and product in $M$ in the following way. Let $(m,i),(m',j)\in M$, then

    \begin{itemize}
        \item $(m,i)+(m',j)=(m+m',\max\{i,j\})$
        \item $(m,i)\cdot (m',j)=(m\cdot m',\max\{i,j\})$
        \item $-(m,i)=(-m,i)$
    \end{itemize}

    It is straightforward to verify that, with the operations above, $M$ is a pre-meadow, but not a pre-meadow with $\abf$. We can turn it into a pre-meadow with $\abf$ by adding the set $\{\abf\}$ and arguing as in example (2) above. So $M'=M\sqcup \{\abf\}$ is a pre-meadow with $\abf$. 
    \end{enumerate}
\end{exmp}

 The following result shows that a pre-meadow with $\abf$ can be turned into a common meadow as long as a certain condition holds. This condition is necessary, as illustrated by Example \ref{Example_revisited}, and amounts to ensure that there is no ambiguity in the choice of the inverses.
    
\begin{theorem}\label{T:Existenceofinverse}
    Let $M$ be a pre-meadow with $\abf$. There exists a unique function $(\cdot)\inv:M\rightarrow M$ such that, for all $x,y\in M$, 
    
        \begin{enumerate}
            \item $x\cdot x\inv = 1+ 0\cdot x\inv$
            \item $(1+0\cdot x)\inv =1+0\cdot x$
            \item $(x\cdot y)\inv=x\inv \cdot y\inv$
            \item $0\inv =\abf$.
        \end{enumerate}

        if and only if for all $x\in M$ the set
         $I_x=\{0\cdot z \in 0\cdot M\mid x\cdot z=1+0\cdot z\} $
        has a unique maximal element.
\end{theorem}
    \begin{proof}        Suppose first that there is a function $(\cdot)\inv:M\rightarrow M$ satisfying $(1)-(4)$. We show that  $0\cdot x\inv$ is the unique maximal element of $I_x$, for all $x\in M$. In order to do that it is enough to show that $0\cdot x\inv \geq 0\cdot z$, for all $0\cdot z\in I_x$.  

        By definition we have that $x\cdot z = 1+0\cdot z.$ Multiplying both sides by $0$ we obtain
        \begin{equation}\label{E:i}\tag{$i$}
            0\cdot (x \cdot z)=0\cdot z.
        \end{equation}

        By properties $(2)$ and $(3)$ of the map $(\cdot)\inv$ we have
        \begin{equation}\tag{$ii$}
            x\cdot z\stackrel{\text{\normalfont\mbox{def}}}{=}1+0\cdot z\stackrel{(2)}{=} (1+0\cdot z)\inv \stackrel{\text{\normalfont\mbox{def}}}{=} (x\cdot z)\inv\stackrel{(3)}{=} x\inv \cdot z\inv.
        \end{equation}
        Then
        \begin{align*}
            0\cdot z\cdot x\inv&\stackrel{(i)}{=} (0\cdot x\cdot z) \cdot x\inv \stackrel{(ii)}{=} 0\cdot (x\inv z\inv) \cdot x\inv = 0\cdot (x\inv \cdot z\inv)  \stackrel{(3)}{=} 0\cdot (x\cdot z)\inv\\
            &\stackrel{\text{\normalfont\mbox{def}}}{=} 0\cdot (1+0\cdot z)\inv \stackrel{(2)}{=}0\cdot (1+0\cdot z) = 0\cdot z.
       \end{align*}

        Hence, $0\cdot x\inv\geq 0\cdot z$, for all $0\cdot z\in I_x$ and therefore $0\cdot x\inv$ is the unique maximal element of $I_x$.

        Assume now that $I_x$ has a unique maximal element, for all $x\in M$. 
        Note that $I_x \neq \emptyset$, since $\abf \in I_x$. 
        We define the function

        $$x\inv := z_x,$$

    \noindent    where $z_x\in M$ is such that $x\cdot z_x=1+0\cdot z_x$ and $0\cdot z_x$ is the unique maximal element of $I_x$. We show that the properties $(1)-(4)$ hold. \\

    $(1)$    From the definition of $I_x$ with $z=x\inv$ it immediately follows that
        $$x\cdot x\inv = 1+ 0\cdot x\inv.$$

    $(2)$   Since $(1+0\cdot x)\cdot (1+0\cdot x)=1+0\cdot (1+0\cdot x)$, we have that $0\cdot x\in I_{1+0\cdot x}$, so all that remains to be seen is that this is indeed the unique maximal element of $I_{1+0\cdot x}$. If $0\cdot z\in I_{1+0\cdot x}$, then
        $$ (1+0\cdot x )\cdot z= 1+0\cdot z, $$
        which implies that $0\cdot x \cdot z =0\cdot z$. Then 
        $$0\cdot (1+0\cdot x)\cdot z = 0\cdot x\cdot z = 0\cdot z.$$

        Hence $0\cdot x \geq 0. \cdot z$, for all $0 \cdot z \in I_{1+0\cdot x}$, which shows that it is the unique maximal element of $I_{1+0\cdot x}$. Moreover,  $(1+0\cdot x)\inv =1+0\cdot x$.\\

        $(3)$ Let $x,y\in M$. Observe that

        \begin{align*}
            (x\cdot y)\cdot (x\inv\cdot y\inv)&=(x\cdot x\inv)\cdot (y\cdot y\inv)\\
            &=(1+0\cdot x\inv)\cdot (1+0\cdot y\inv)=1+0\cdot x\inv \cdot y\inv,
        \end{align*}

       \noindent and so $0\cdot x\inv y\inv \in I_{x\cdot y}$. We show that $0.x\inv\cdot y\inv$ is the unique maximal element of $I_{x\cdot y}$. Let $0\cdot w\in I_{x\cdot y}$. Then, by definition,
        
        $$(x\cdot y)\cdot w= 1+0\cdot w.$$

        Multiplying both sides by $0$ we obtain

        \begin{equation}\tag{$iii$}
            0\cdot (x \cdot y )\cdot w = 0\cdot w.
        \end{equation}

        Then 
        \begin{equation}\tag{$iv$}
        0\cdot w\cdot y \stackrel{(iii)}{=} 0\cdot (x \cdot y \cdot w) \cdot y = 0\cdot x \cdot y \cdot w= \stackrel{(iii)}{=} 0 \cdot w.
        \end{equation}

        Hence 
        $$x\cdot(y\cdot w)=1+0\cdot w= 1+0\cdot (y\cdot w),$$

        and therefore we have that $0\cdot (y\cdot w)\in I_x$. With a similar argument one shows that $0\cdot (x\cdot w)\in I_y$. Since, by definition, $x\inv$ and $y\inv$ are the maximal elements of $I_x$ and $I_y$, respectively, we have

        \begin{equation}\tag{$v$}
        \begin{split}
            0\cdot x\inv (y\cdot w)&=0\cdot (y\cdot w) \stackrel{(iv)}{=}0\cdot w \mbox{, and} \\
             0\cdot y\inv (x\cdot w)&=0\cdot (x\cdot w) \stackrel{\phantom{(iv)}}{=} 0\cdot w.
            \end{split}
        \end{equation}
        
        Also, by property $(1)$ we have

        \begin{equation}\tag{$vi$}
            0\cdot x\cdot x\inv =0\cdot x\inv.
        \end{equation}

        We are now able to show that $0\cdot x\inv \cdot y\inv$ is the maximal element of $I_{x\cdot y}$. Indeed,

        \begin{align*}
        0\cdot (x\inv\cdot y\inv ) \cdot w &\stackrel{(vi)}{=} 0\cdot (x\cdot x\inv) \cdot y\inv \cdot w= (0\cdot y\inv \cdot x\cdot w)\cdot x\inv=\\
        &\stackrel{(v)}=0\cdot w \cdot x\inv \stackrel{(iv)}{=} 0\cdot (y\cdot w)\cdot x\inv \stackrel{(v)}{=}0\cdot w.
        \end{align*}

        Then $0\cdot (x\inv \cdot y\inv)$ is the unique maximal element of $I_{x\cdot y}$ and  $(x\cdot y)\inv =x\inv \cdot y\inv$.\\
        
       $(4)$  Let $0\cdot x\in I_0$. By the definition of $I_0$, 
            $$0\cdot x = 1+0\cdot x.$$
        Then the unit of the ring $M_{0\cdot x}$ is equal to its zero (see Theorem \ref{T:union}) and so $M_{0\cdot x}$ must be the trivial ring $\{\abf\}$. Then $0\cdot x =\abf$, which means that $I_0=\{\abf\}$ and therefore $0\inv =\abf$.

        In order to show unicity, let $x\mapsto \frac{1}{x}$ be another function satisfying the properties $(1)-(4)$.

            Let $x\in M$. Then we have
            \begin{align}\label{E:Div1}
                \frac{1}{x}\cdot \frac{1}{x\inv} & = \frac{1}{x\cdot x\inv}=\frac{1}{1+0\cdot x\inv} = 1+0\cdot x\inv =x\cdot x\inv.
            \end{align}

            Hence $\frac{1}{x}\cdot \frac{1}{x\inv} =x\cdot x\inv$.
            We also have that
            \begin{align}\label{E:Div2}
                0\cdot x\inv =& 0\cdot 1 + 0\cdot 0\cdot x\inv=0\cdot(1+0\cdot x\inv)  \\
                &= 0\cdot (x \cdot x\inv) = 0\cdot \left(\frac{1}{x}\cdot \frac{1}{x\inv}\right) = 0\cdot \left(1+0\cdot \frac{1}{x}\right)=0\cdot \frac{1}{x}.
            \end{align}
            We conclude that $0\cdot x\inv=0\cdot \frac{1}{x}$.
            Now we have that
            \begin{align*}
                x\cdot x\inv = 1+ 0\cdot x\inv = 1+0\cdot \frac{1}{x}=x\cdot \frac{1}{x}.
            \end{align*}
            Hence $x\cdot x\inv =x\cdot \frac{1}{x}$.
            Combining \eqref{E:Div1} and \eqref{E:Div2} we have
            $$x\inv=(x\cdot x\inv) \cdot x\inv=\left(x\cdot \frac{1}{x}\right)\cdot x\inv = (x\cdot x\inv)\cdot \frac{1}{x} = \left(x\cdot \frac{1}{x}\right)\cdot \frac{1}{x} = \frac{1}{x}.$$
        Hence $x\inv = \frac{1}{x}$.
        
    \end{proof}

\begin{definition}
    Let $M$ be a pre-meadow with $\abf$. If for all $x\in M$ the set 
    $$ I_x=\{0\cdot z \in 0\cdot M\mid x\cdot z=1+0\cdot z\}$$
   has a unique maximal element, we say that $M$ is an \emph{invertible pre-meadow}.
\end{definition}

We observe that, since $0\cdot M$ is not totally ordered, it is possible for $I_x$ to have more than one maximal element, this is illustrated in Example \ref{Example_revisited} $(3)$ below.

\begin{exmp}\label{Example_revisited}
\begin{enumerate}
    \item In the pre-meadow with $\abf$, $M=R\sqcup\{\abf\}$, where $R$ is a unital commutative ring, we have that  $I_x=\{0,\abf\}$ if $x$ is invertible in $R$, and in this case $x\inv$ is the inverse of $x$ in $R$;  if $x$ is not invertible in $R$ then $I_x=\{\abf\}$ and the inverse of $x$ is $\abf$.
    \item In the pre-meadow with $\abf$, $M=\z\times\{0\}\sqcup\q\times\{1\}\sqcup \{\abf\}$ we have that for any $(m,0)\in \z\times \{0\}$, with $m\notin \{-1,0,1\}$, it holds that $I_{(m,0)}=\{(0,1),\abf\}$, and so the inverse of $(m,0)$ is $(\frac{1}{m},1)$.
    \item Consider $M=\z\times \{0\}\sqcup \q\times \{1\}\sqcup \q \times \{2\}\sqcup \{\abf\}$ with the operations
    $$ (x,i)+(y,j)=
\begin{cases}
(x+ y,i),&\;i=j \text{ or } j=0\\
(x+ y,j),&\;i=0\\
\abf,&\;\text{otherwise}
\end{cases} $$
and
$$ (x,i)\cdot(y,j)=
\begin{cases}
(x\cdot y,i),&\;i=j \text{ or } j=0\\
(x\cdot y,j),&\;i=0\\
\abf,&\;\text{otherwise}
\end{cases} $$

We have that $M$ is a pre-meadow with $\abf$, but not an invertible pre-meadow.     
If an inverse function existed, then in $M\setminus \{(-1,0),(0,0),(1,0)\}$ the inverse of $(x,0)$ would be either $(\frac{1}{x},i)$ or $\abf$. However, $i$ must be different from $0$ because otherwise $\frac{1}{m}\in \z$ and $m=-1$ or $1$. If, on the other hand, $i=1$, then
$$\left((x,0)\cdot\left(\frac{1}{x},2\right)\right)\inv= (1,2)\inv=(1,2)$$

and, from Property $(3)$, 
$$\left((x,0)\cdot
\left(\frac{1}{x},2\right)\right)\inv= (x,0)\inv\cdot\left(\frac{1}{x},2\right)\inv=\left(\frac{1}{x},1\right)\cdot(x,2)=\abf. $$
which would imply $\abf=(1,2)$. A similar argument holds if $i=2$ or if the inverse would be $\abf$. 

An alternative way to show that $M$ is not an invertible pre-meadow is to show that, for example, the set $I_{(2,0)}=\{(0,1),(0,2),\abf\}$ has both $(0,1)$ and $(0,2)$ as maximal elements.
\end{enumerate}
\end{exmp}

The following result covers some basic properties of invertible pre-meadows. The proofs are essentially the same as in \cite[Propositions 2.2.1 and 2.3.1]{Bergstra2015}.

    \begin{proposition}\label{P:Identities}
        Let $M$ be a invertible pre-meadow. Then
        \begin{enumerate} 
            \item $(x \cdot x\inv ) \cdot x\inv= x\inv$
            \item $(-x)\inv = -(x\inv)$ 
            \item $(x \cdot x\inv)\inv = x \cdot x\inv$
            \item $(x\inv)\inv = x + 0 \cdot x\inv$
            \item $x \cdot \abf =-\abf =  \abf\inv=\abf$
            \item $0\cdot x = \abf\rightarrow x= \abf$ 
            \item $0\cdot x\cdot y = 0 \rightarrow 0\cdot x = 0$.
        \end{enumerate}
    \end{proposition}

Let us now recall the notion of common meadow,  introduced in \cite{Bergstra2015}. 

\begin{definition}
    A \emph{common meadow} is a structure $(M,+,-,\cdot)$ satisfying the following equations
    
\begin{tabular}{l c c r}
\\[-2mm]
$(\M_1)$ &\qquad \qquad \qquad&$(x+y)+z=x+(y+z) $  & \qquad \qquad \\[2mm]
$(\M_2)$ && $x+y=y+x $ \\[2mm]
$(\M_3)$ && $x+0=x$  \\[2mm]
$(\M_4)$ && $x+ (-x)=0 \cdot x$ \\[2mm]
$(\M_5)$ && $(x \cdot y) \cdot z=x \cdot (y \cdot z)$  \\[2mm]
$(\M_6)$ && $x \cdot y=y \cdot x $ \\[2mm]
$(\M_7)$ && $1 \cdot x=x$ \\[2mm]
$(\M_8)$ && $x \cdot (y+z)= x \cdot y + x \cdot z$ \\[2mm]
$(\M_9)$ && $-(-x)=x$ \\[2mm]
$(\M_{10})$ && $x \cdot x^{-1}=1 + 0 \cdot x^{-1}$\\[2mm]
$(\M_{11})$ && $(x \cdot y)^{-1} = x^{-1} \cdot y^{-1}$\\[2mm]
$(\M_{12})$ && $(1 + 0 \cdot x)^{-1} = 1 + 0 \cdot x $\\[2mm]
$(\M_{13})$ && $ 0^{-1}=\abf$\\[2mm]
$(\M_{14})$ && $x+ \abf = \abf $\\[2mm]
\end{tabular}
\end{definition}

Note that in a common meadow we have that $0\cdot(x+y)=0\cdot x \cdot y$ (see \cite[Proposition~2.2.1]{Bergstra2015}). The following result then easily follows from Theorem \ref{T:Existenceofinverse}.

\begin{corollary}\label{C:precomon}
    A pre-meadow $M$ is an invertible pre-meadow if and only if it is a common meadow.
\end{corollary}
    
With the previous result in mind we will refer to invertible pre-meadows as "common meadows", or simply as "meadows", since all the meadows considered below are common meadows.

\begin{proposition}\label{P:AssemblyProduct}
    Let $M$ be a meadow. Let $U(M):=\{x\in M\mid 0\cdot x = 0\cdot x\inv\}$. Define $u(x)=1+0\cdot x$ and $d(x)=x\inv$, for all $x\in U(M)$. Then the structure $(P,\cdot,u,d)$ is an assembly.
\end{proposition}
\begin{proof}
    First note that if $x,y\in U(M)$, then 
    $$0\cdot (x\cdot y)\inv=0\cdot x\inv \cdot y\inv = 0\cdot x\cdot y$$
    Then $U(M)$ is closed under the product.
    For $x\in U(M)$, we have
    $$u(x)\cdot x=(1+0\cdot x)\cdot x=x+0\cdot x\cdot x = x+0\cdot x = x$$
    and if $y\cdot x = x$ we have 
    $$y\cdot (1+0\cdot x) = y\cdot (1+0\cdot x\inv) \cdot y\cdot x\cdot x\inv = x\cdot x\inv= (1+0\cdot x\inv)=(1+0\cdot x).  $$ 
    Hence $(A_1)$ is satisfied.

    For $(A_2)$, note that $$x \cdot d(x)=x\cdot x\inv=1+0\cdot x\inv=1+0\cdot x=u(x),$$ and $$u(d(x))=1+0\cdot x\inv = 1+0\cdot x= u(x).$$

    Finally, for $(A_3)$ we have 
    \begin{equation*}
        \begin{split}    
        u(x\cdot y)&=1+0\cdot x \cdot y= 1+0\cdot x \cdot y+0\cdot x+0 \cdot y\\
        &= (1+0\cdot x)\cdot(1+0\cdot y)=u(x)\cdot u(y). \qedhere
        \end{split}
        \end{equation*}
\end{proof}

Recall that if $R$ is a commutative ring then $(R,+)$ is an abelian group, and the group of the units of $R$ is also an abelian group. In meadows we can replace abelian groups with commutative assemblies. By Proposition \ref{P:Assembly}, if $M$ is a meadow then $(M,+)$ is an commutative assembly, and by Proposition \ref{P:AssemblyProduct}, the set $\{x\in M\mid 0\cdot x = 0\cdot x\inv\}$ with the product is also a commutative assembly. This suggests a relation with an algebraic structure called \emph{association} (see \cite{DinisBerg(17),DinisBerg(book)}). In fact, it is an easy consequence of Propositions \ref{P:Assembly} and \ref{P:AssemblyProduct}  that a common meadow is always an association. However, we would like to point out that the converse is not true in general.
 
\section{Directed Lattices}\label{S:DirectedLattices}

Recall that, from Theorem \ref{T:union}, if $M$ is a pre-meadow then $M=\bigsqcup_{z\in 0\cdot M}M_z$, where each $M_z$ is a ring. Additionally, from Proposition \ref{P:order}  we have that $0\cdot M$ has a partial order, and so the partial order on pre-meadows induces a partial order on the set of rings $\{M_z\mid z\in 0\cdot M\}$. We will see next that it is possible to turn such partial order into a  directed graph.

\begin{proposition}\label{P:Transitionmaps}
            Let $M$ be a meadow. If $z,z'\in 0\cdot M$ are such that $z\leq z'$, then the map
            \begin{align*}
                f_{z,z'}:M_{z'}&\rightarrow M_z\\
                        x&\mapsto x+z
            \end{align*}
            is a ring homomorphism.

            Moreover, if $z,z',z''\in 0\cdot M$ are such that $z\leq z'\leq z''$, then $f_{z,z'}\circ f_{z',z''}=f_{z,z''}$.
        \end{proposition}
        \begin{proof}
            We start by showing that the map is well defined. For $x\in M_{z'}$ we have that $x=0 \cdot z'$ and $z \cdot z'=z$, since $z \leq z'$. Note also that $z$ is idempotent, since $z \in 0 \cdot M$. Then
            \[(x+z)\cdot 0=0\cdot x + 0\cdot z = 0\cdot z' + 0\cdot z =  0\cdot z' + 0\cdot z\cdot z'=(0+0\cdot z)\cdot z'=0\cdot z\cdot z' =0\cdot z. \]
            Hence $x+z \in M_z$.

            Now, let $x,y\in M_{z'}$. Then $z'=(x+y)\cdot z'$, and
            \begin{itemize}
                \item $f_{z,z'}(x+y)=x+y+z=x+z+y+z=f_{z,z'}(x)+f_{z,z'}(y)$.
                \item $f_{z,z'}(x\cdot y)=x \cdot y+z=x\cdot y+z+z=x\cdot y+z'\cdot z+z=x\cdot y+(x+y)\cdot z'\cdot z+z=x\cdot y+(x+y)\cdot z+z\cdot z =(x+z)\cdot (y+z)=f_{z,z'}(x) \cdot f_{z,z'}(y)$.
                \item $f_{z,z'}(1+z')=1+z'+z=1+z'+z\cdot z'=1+(1+z) \cdot z'=1+(1+z)\cdot 0 \cdot z'=1+z$.
            \end{itemize}
            Hence $f_{z,z'}$ is a ring homomorphism.

            Finally, take $z,z',z''\in 0\cdot M$ such that $z\leq z'\leq z''$. Recall that $z'+z = z'\cdot z=z$. Then, for $x\in M_{z''}$ we have
            \[f_{z,z'}\circ f_{z',z''}(x)=x+z'+z=x+z=f_{z,z''}(x).\qedhere\]
        \end{proof}

    Proposition \ref{P:Transitionmaps} suggests a relation between  meadows and commutative diagrams of rings. In fact, from Corollary \ref{C:precomon} one deduces that if $M$ is a meadow then there is a partial order relation in $\{M_z\mid z\in 0\cdot M\}$ that defines a lattice. So, to each meadow we can associate a commutative ring that has a lattice structure. With that in mind we give the following definition.

    \begin{definition}\label{D:LatticeRings}
        A \emph{directed lattice} of rings $\Gamma$ over a countable lattice $L$ consists on a family of commutative rings $\Gamma_i$ indexed by $i\in L$, such that $\Gamma_i$ is a unital commutative ring for all $i\in L\setminus \min(L)$ and $\Gamma_{\min(L)}$ is the zero ring, together with a family of ring homomorphisms $f_{j,i}:\Gamma_i\rightarrow\Gamma_j$ whenever $i>j$ such that $f_{j,k}\circ f_{i,j}=f_{i,k} $ for all $i>j>k$. 
    \end{definition}

    From Proposition \ref{P:Transitionmaps} we see that every common meadow $M$ can be associated with a directed lattice of rings over the lattice $0\cdot M$. The next result shows that the converse also holds. That is, we can associate a common meadow to every directed lattice of rings over a lattice. We denote by $R^{\times}$ the set of the invertible elements of the ring $R$.
    
    \begin{theorem}\label{T:DirectedLattice}
        Let $L$ be a lattice, and $\Gamma$ a directed lattice of rings over  $L$ such that, for all $i\in I$ and all $x\in \Gamma_i$ the set:
        $$J_x=\{j\in I\mid f_{j,i}(x)\in \Gamma_j^{\times}\}$$
        has a unique maximal element.
          Then, there exists a meadow $M=\bigsqcup_{i\in L}\Gamma_i$, such that the lattice $0\cdot M$ is equivalent to $L$.
        \end{theorem}
            
        \begin{proof}
            Let $M=\bigsqcup_{i\in L}\Gamma_i$ be the disjoint union of the rings $\Gamma_i$. 
            Since $L$ is a lattice there exists a unique maximal element $M_0:=\max(L)$.            
             We denote the meet of $i,j\in L$ by $i\wedge j$. Let $x\in \Gamma_i$ and $y\in \Gamma_j$. The operations in $M$ are defined as follows
            \begin{itemize}
                \item $x+y=f_{i\wedge j,i}(x)+_{i\wedge j}f_{i\wedge j,j}(y)$, where $+_{i\wedge j}$ is the sum in $\Gamma_{i\wedge j}$;
                \item $x \cdot y=f_{i\wedge j,i}(x)\cdot_{i\wedge j}f_{i\wedge j,j}(y)$, where $\cdot_{i\wedge j}$ is the product in $\Gamma_{i\wedge j}$.
            \end{itemize}
            We start by showing that the operations are associative, i.e.\ that $M$ satisfies $(\PM_1)$ and $(\PM_5)$. 

            Let $x\in \Gamma_i$, $y\in \Gamma_j$ and $z\in \Gamma_k$. Then 
            \begin{align*}
                x+(y+z) &= x+(f_{j\wedge k, j}(y) + f_{j\wedge k, k}(z))\\
                &= f_{i\wedge j \wedge k,i}(x)+f_{i\wedge j \wedge k,j\wedge k}(f_{j\wedge k, j}(y) + f_{j\wedge k, k}(z))\\
                &=  f_{i\wedge j \wedge k,i}(x)+f_{i\wedge j \wedge k,j\wedge k}\circ f_{j\wedge k, j}(y) + f_{i\wedge j \wedge k,j\wedge k}\circ f_{j\wedge k, k}(z)\\
                &= f_{i\wedge j \wedge k,i}(x)+f_{i\wedge j \wedge k,j}(y) + f_{i\wedge j \wedge k,k}(z)\\
                &=  f_{i\wedge j \wedge k,i\wedge j}\circ f_{i\wedge j,i}(x)+f_{i\wedge j \wedge k,i\wedge j}\circ f_{i\wedge j,j}(y) 
                + f_{i\wedge j \wedge k,k}(z)\\
                & = f_{i\wedge j \wedge k,i\wedge j}( f_{i\wedge j,i}(x)+f_{i\wedge j,j}(y) )
                + f_{i\wedge j \wedge k,k}(z) = (x+y) + z.\\
            \end{align*}
            The proof that $\cdot$ is associative is analogous. 

            One easily sees that the operations $+$ and $\cdot$ are commutative, and so axioms $(\PM_2)$ and $(\PM_6)$ hold in $M$.

            Let $0_{M_0}$ and $1_{M_0}\in \Gamma_{M_0}$ be the zero and the identity of the ring $\Gamma_{M_0}$ respectively. In this case, we drop the index and write simply $0=0_{M_0}$ and $1=1_{M_0}$. We show that these elements are the zero and the identity of $M$, thus showing that $(\PM_3)$ and $(\PM_7)$ hold in $M$. For $x\in \Gamma_i$, and denoting the zero of $\Gamma_i$ by $0_i$ and its unit by $1_i$, we have
            \begin{align*}
                0+x&=f_{i,{M_0}}(0)+f_{i,i}(x) = 0_i+x = x, \text{ and}\\
                1\cdot x&=f_{i,{M_0}}(1) \cdot f_{i,i}(x) = 1_i \cdot x = x.
            \end{align*}
        
        Since $\Gamma_i$ is a ring, for every $x\in \Gamma_i$ there exists $-x\in\Gamma_i$. 
        Let us now see that $M$ satisfies axiom $(\PM_4)$. Let $x\in \Gamma_i$. Then
            $$x+(-x)=0_i=0_i\cdot x=f_{i,{M_0}}(0)\cdot f_{i,i}(x)=0 \cdot x.$$

        Hence $(\PM_4)$ holds. From the definition of $-x$ it follows immediately that $(\PM_9)$ holds.

        Finally, we check that the distributive law 
 $(\PM_8)$ holds in $M$. Let $x\in \Gamma_i$, $y\in \Gamma_j$ and $z\in \Gamma_k$. Then

            \begin{align*}
                x \cdot& (y+z)\\
                &= x\cdot (f_{j\wedge k, j}(y) + f_{j\wedge k, k}(z))\\ &= f_{i\wedge j \wedge k,i}(x)\cdot f_{i\wedge j \wedge k,j\wedge k}(f_{j\wedge k, j}(y) + f_{j\wedge k, k}(z))\\
                &= f_{i\wedge j \wedge k,i}(x)\cdot f_{i\wedge j \wedge k,j\wedge k}\circ f_{j\wedge k, j}(y) + f_{i\wedge j \wedge k,i}(x)\cdot f_{i\wedge j \wedge k,j\wedge k}\circ f_{j\wedge k, k}(z)\\
                &=f_{i\wedge j \wedge k,i}(x)\cdot f_{i\wedge j \wedge k,j}(y) + f_{i\wedge j \wedge k,i}(x)\cdot f_{i\wedge j \wedge k,k}(z)\\
                &=f_{i\wedge j \wedge k,i\wedge j}(f_{i\wedge j ,i}(x)\cdot f_{i\wedge j ,j}(y)) + f_{i\wedge j \wedge k,i\wedge k}(f_{i\wedge k ,i}(x)\cdot f_{i\wedge k ,k}(z))\\ &=x\cdot y+x\cdot z.\\
            \end{align*}
            Hence $M$ is a pre-meadow. 
            
            
            Note that for all $x\in \Gamma_i$, one has $0 \cdot x=f_{i,{M_0}}(0)\cdot f_{i,i}(0) = 0_i$. Then $0 \cdot M=\{0_i\mid i\in L\}$, and since $L$ is a lattice it follows that $0\cdot M$ also has the same order as $L$. It is straightforward to see that the order defined in Definition \ref{D:order} is the same order as $L$. Then, condition $3$ in Definition \ref{D:PreMeadowA} is verified. From Definition \ref{D:LatticeRings} it follows that conditions $1$ and $4$ are also verified. Recall that if $m=min(L)$, then $\Gamma_m$ is the zero ring. We shall denote the unique element of $\Gamma_m$ by $\abf$.
            
            Let us see that condition $2$ of Definition \ref{D:PreMeadowA} also holds. Let $x\in \Gamma_i$. Then
            $$x+\abf = f_{m,i}(x)+\abf = \abf.$$
            Hence $M$ is a pre-meadow with $\abf$.

            Finally, for $x\in \Gamma_i\subseteq M$ we have that if $j\in J_x$ then there exists an $\frac{1}{x}\in \Gamma_j$ such that $$x\cdot\frac{1}{x}=1_j=1_j+0_j \cdot\frac{1}{x}= 1+ 0\cdot \frac{1}{x}.$$

            Then $f_{j,M_0}(0)=0_j=0\cdot \frac{1}{x}\in I_x$.

            Conversely, we have that if $0\cdot z\in I_x$ is such that $x\cdot z = 1+ 0\cdot z$, where $0\cdot z = f_{j,M_0}(0)$ for some $j\in I$, then $f_{j,i}(x)$ is invertible with inverse $z\in \Gamma_j$. 

            Lastly note that $f_{j,M_0}(0)\cdot f_{j',M_0}(0)= f_{j,M_0}(0)$ if and only if $j\geq j'$. Then $I_x$ and $J_x$ are isomorphic as  partial ordered sets. We conclude that $I_x$ has a unique maximal element and, by Theorem \ref{C:precomon}, $M$ is a common meadow.
        \end{proof}

The following result is a consequence of the proof of Theorem~\ref{T:DirectedLattice}.

    \begin{corollary}\label{C:DirectedLattice}
    Let $L$ be a lattice and $\Gamma$ a directed lattice of rings over $L$. Then, there exists $M=\sqcup_{i\in I}\Gamma_i$, a pre-meadow with $\abf$ such that the lattice $0\cdot M$ is equivalent to $L$.
    \end{corollary}
 \begin{remark} 
        Let $M$ be a pre-meadow with $\abf$. For all $0\cdot z\in 0\cdot M$ we have $(0\cdot z)\cdot \abf = \abf$, that is, $\abf\leq 0\cdot z$. Then the lattice $0\cdot M$ has a minimum equal to $\abf$.
    \end{remark}
    
   We may use Theorem~\ref{T:DirectedLattice} to construct meadows, using the operations induced by a lattice. We present some examples of such constructions.
   
   \begin{exmp}
       Consider the lattice
       
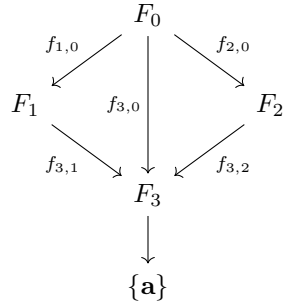
\[\begin{tikzcd}[ampersand replacement=\&]
	\& {F_0} \\
	{F_1} \&\& {F_2} \\
	\& {F_3} \\
	\& {\{\abf\}}
	\arrow["{f_{1,0}}"', from=1-2, to=2-1]
	\arrow["{f_{2,0}}", from=1-2, to=2-3]
	\arrow["{f_{3,1}}"', from=2-1, to=3-2]
	\arrow["{f_{3,2}}", from=2-3, to=3-2]
	\arrow[from=3-2, to=4-2]
	\arrow["{f_{3,0}}"', from=1-2, to=3-2]
\end{tikzcd}\]
where $F_0,F_1,F_2,F_3$ are fields and $f_{i,j}$ are ring homomorphisms from $F_j$ to $F_i$. Then by Theorem \ref{T:DirectedLattice} there is a common meadow $M=F_0\sqcup F_1\sqcup F_2\sqcup F_3\sqcup \{\abf\}$. Following the proof of Theorem \ref{T:DirectedLattice} we see that the product (and sum) of elements of $M$ that are in $F_0$ is just the product (and sum) defined in $F_0$. In order to calculate the product of an element $x\in F_1$ with another element $y\in F_2$ one first "sends" the elements to $F_3$, i.e.\ calculate $f_{3,1}(x)$ and $f_{3,2}(y)$, and then multiplies them in $F_3$. 

\end{exmp}     
    \begin{exmp}
        Consider the lattice

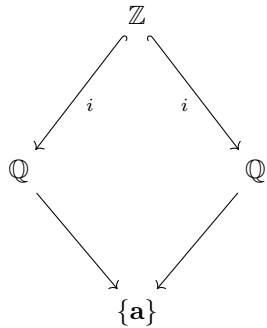
\[\begin{tikzcd}
	& \z \\
	& {} && {} \\
	\q && \q \\
	\\
	& {\{\abf\}}
	\arrow["i", hook, from=1-2, to=3-1]
	\arrow["i"', hook', from=1-2, to=3-3]
	\arrow[from=3-3, to=5-2]
	\arrow[from=3-1, to=5-2]
\end{tikzcd}\]
where $i:\z\rightarrow \q$ is the inclusion map. The previous diagram corresponds to the pre-meadow with $\abf$ of Example \ref{firstexmp} (3), so $M=\z\sqcup\q\sqcup\q\sqcup \{\abf\}$ with the operations defined by the lattice is a pre-meadow with $\abf$ that is not a common meadow. Note that otherwise there would be some ambiguity in choosing an inverse for, say, $2 \in \z$.  
\end{exmp}
\begin{exmp}

Consider the lattice 
\[\begin{tikzcd}
	\z \\
	\\
	\q \\
	\\
	{\{\abf\}}
	\arrow["i"', hook, from=1-1, to=3-1]
	\arrow[from=3-1, to=5-1]
\end{tikzcd}\]

The set $N=\z\sqcup \q \sqcup \{\abf\}$ is a common meadow, where the inverse of an element $m\in \z\setminus \{-1,0,1\}$ is $\frac{1}{m}\in\q$. Note that this example is the same as the common meadow  in Example \ref{firstexmp} (3).

Even though $\z\subseteq\q$, in the disjoint unions above these sets are "disjoint", i.e.\ there is another copy of $\z$ in $\q$.
    \end{exmp}
    We can also use common meadows in order to represent commutative diagrams as exemplified below.

\begin{exmp}
 Take the following commutative diagram of rings
    \[\begin{tikzcd}[ampersand replacement=\&]
	{R_1} \& {R_2} \& {R_3} \\
	{R'_1} \& {R'_2} \& {R'_3}
	\arrow["f", from=1-1, to=1-2]
	\arrow["g", from=1-2, to=1-3]
	\arrow["{f'}", from=2-1, to=2-2]
	\arrow["{g'}", from=2-2, to=2-3]
	\arrow["{\phi_1}"', from=1-1, to=2-1]
	\arrow["{\phi_2}"', from=1-2, to=2-2]
	\arrow["{\phi_3}"', from=1-3, to=2-3]
\end{tikzcd}\]

We can transform this commutative diagram into the following lattice

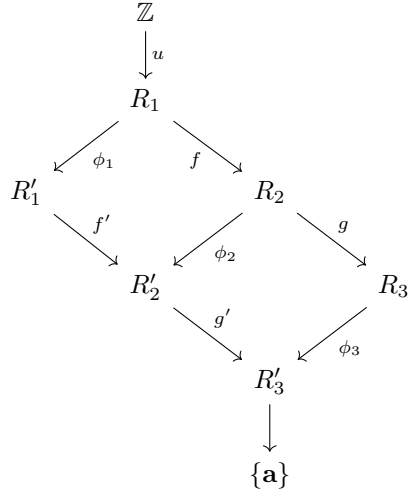
\[\begin{tikzcd}[ampersand replacement=\&]
	\& {\mathbb{Z}} \\
	\& {R_1} \\
	{R'_1} \&\& {R_2} \\
	\& {R'_2} \&\& {R_3} \\
	\&\& {R'_3} \\
	\&\& {\{\abf\}}
	\arrow["f"', from=2-2, to=3-3]
	\arrow["g", from=3-3, to=4-4]
	\arrow["{\phi_1}", from=2-2, to=3-1]
	\arrow["{f'}", from=3-1, to=4-2]
	\arrow["{g'}", from=4-2, to=5-3]
	\arrow[from=5-3, to=6-3]
	\arrow["{\phi_3}", from=4-4, to=5-3]
	\arrow["{\phi_2}", from=3-3, to=4-2]
	\arrow["u", from=1-2, to=2-2]
\end{tikzcd}\]
where the map $u:\z\rightarrow R_1$ is the map $u(m)=m\cdot 1_{R_1}$. 

As mentioned before, for the directed lattice to respect the condition of Theorem~\ref{T:DirectedLattice} there can be no ambiguity in the choice of the inverse. For example, an element of $R_1$ that is not invertible, cannot be invertible in both $R'_1$ and $R_2$. Provided that this restriction is satisfied, Theorem \ref{T:DirectedLattice} ensures that there is a common meadow associated with the lattice above. 
\end{exmp}

\begin{exmp}
    Consider the following directed lattices, where the homomorphisms are the natural homomorphisms and $p$ is a prime number. 

    
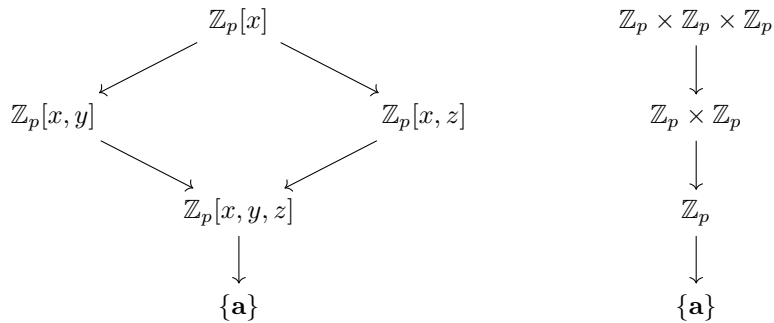
\[\begin{tikzcd}
	& {\mathbb{Z}_p[x]} &&& {\mathbb{Z}_p\times \mathbb{Z}_p\times \mathbb{Z}_p} \\
	{\mathbb{Z}_p[x,y]} && {\mathbb{Z}_p[x,z]} && {\mathbb{Z}_p\times \mathbb{Z}_p} \\
	& {\mathbb{Z}_p[x,y,z]} &&& {\mathbb{Z}_p} \\
	& {\{\abf\}} &&& {\{\abf\}}
	\arrow[from=1-2, to=2-1]
	\arrow[from=1-2, to=2-3]
	\arrow[from=2-3, to=3-2]
	\arrow[from=2-1, to=3-2]
	\arrow[from=3-2, to=4-2]
	\arrow[from=1-5, to=2-5]
	\arrow[from=2-5, to=3-5]
	\arrow[from=3-5, to=4-5]
\end{tikzcd}\]

 It is easy to see that these diagrams define meadows. From these meadows one can create a new meadow by gluing the directed lattices as follows
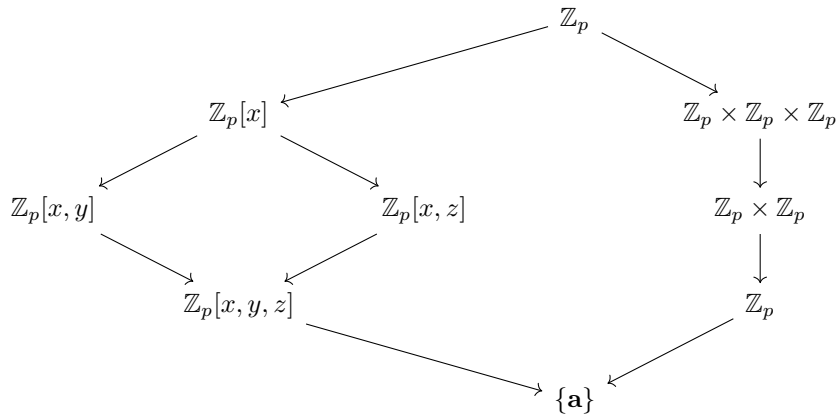
\[\begin{tikzcd}
	&&& {\mathbb{Z}_p} \\
	& {\mathbb{Z}_p[x]} &&& {\mathbb{Z}_p\times \mathbb{Z}_p\times \mathbb{Z}_p} \\
	{\mathbb{Z}_p[x,y]} && {\mathbb{Z}_p[x,z]} && {\mathbb{Z}_p\times \mathbb{Z}_p} \\
	& {\mathbb{Z}_p[x,y,z]} &&& {\mathbb{Z}_p} \\
	&&& {\{\abf\}}
	\arrow[from=2-2, to=3-1]
	\arrow[from=2-2, to=3-3]
	\arrow[from=3-3, to=4-2]
	\arrow[from=3-1, to=4-2]
	\arrow[from=2-5, to=3-5]
	\arrow[from=3-5, to=4-5]
	\arrow[from=1-4, to=2-5]
	\arrow[from=1-4, to=2-2]
	\arrow[from=4-5, to=5-4]
	\arrow[from=4-2, to=5-4]
\end{tikzcd}\]
\end{exmp}
Inspired by the previous example, we have a new process of creating meadows by joining two existing meadows (note that if we take an $x$ in the "top" $\z_p$ and it is different from zero it is invertible in $\z_p$, and if we take $x$ outside of $\z_p$, then $J'_x$ in this new lattice is equal to $J_x$ in the old lattice and then it satisfies the condition of meadow).  Given two meadows $M$ and $N$, such that $p\cdot 1_M=0_M$ and $p\cdot 1_N=0_N$, where $p$ is a prime number, we can define a new meadow $M+_{\z_p} N := M\sqcup N\setminus\{\abf\}\sqcup \z_p$, where the corresponding lattice is given by joining the maximum of both lattices with the (unique) homomorphism $f:\z_p\rightarrow M_0$ and the (unique) homomorphism $g:\z_p\rightarrow N_0$. It is easy to see that such construction gives rise to a meadow, where for $x\in M$ and $y\in N$ we have $x+y=x\cdot y=\abf$. Note that here we identify $\abf_M$ and $\abf_N$ with a unique $\abf$. 

Recall that if $M$ is a common meadow, then the set $0\cdot M$ has a partial order defined by $0\cdot x\leq 0\cdot y$ if $0 \cdot x\cdot y = 0\cdot x$. And the partial order is \emph{total} if and only if for all $x,y\in M$ we have $0\cdot x\cdot y = 0\cdot x$ or $0\cdot x\cdot y = 0\cdot y$. The following proposition summarizes this discussion.

\begin{proposition}
    Let $M$ be a common meadow. Then the partial order in $0\cdot M$ is a total order if and only if the assembly $(M,+,e,s)$ (with $e$ and $s$ defined as in Proposition \ref{P:Assembly}) is a strong assembly.
\end{proposition}




\section{Algebraic constructions on Meadows}\label{S:Algebraicconstructions}

By relating common meadows with lattices, Theorem \ref{T:DirectedLattice} promptly allows to extend the usual algebraic notions to the context of common meadows. In this section we make this more explicit for the notions of homomorphism, ideals, kernels and isomorphisms. 

\subsection{Homomorphisms of Meadows}

In analogy with the similar notion in rings, an \emph{homomorphism of common meadows} is a map which is linear for both addition and multiplication, and maps the element $1_M$ to $1_{N}$. 
\begin{definition}\label{D:Morphism}
    Let $f:M\rightarrow N$ be a function. We say that $f$ is an \emph{homomorphism of (common) meadows} if $M,N$ are common meadows and for all $x,y\in M$
    \begin{enumerate}
        \item $f(x+y)=f(x)+f(y)$.
        \item $f(x\cdot y)=f(x) \cdot f(y)$.
        \item $f(1_M)=1_{N}$.        
    \end{enumerate}
\end{definition}

\begin{exmp}\label{E:RingHom}
    Let $R_0$ and $R_1$ be commutative unital rings and let $f:R_0\rightarrow R_1$ be a ring homomorphism. Then, as in Example \ref{firstexmp} $(1)$, we can define the meadows $M=R_0\sqcup\{\abf\}$ and $N=R_1\sqcup\{\abf\}$. One easily sees that the map $\bar{f}:M\rightarrow N$ defined by $\bar{f}(x)=f(x)$ if $x\in R_0$ and $\bar{f}(\abf)=\abf$ is an homomorphism of meadows.
\end{exmp}

     In Example \ref{E:RingHom},
in the particular case where $R_0=\z$, $R_1=\q$ and $f$ is the inclusion homomorphism of $\z$ into $\q$, the common meadow homomorphism does not commute with the inverse. To see this take, for example, $5\in\z$ whose inverse in $M$ is $\abf$, while the inverse of $5$ in $N$ is $\frac{1}{5}$.
For this reason, in the definition of homomorphism of meadows we are not requiring the inverse to commute with the function, i.e.\ that $f(x\inv)=f(x)\inv$.


    
 The following proposition summarizes some basic results concerning homomorphisms of meadows.
\begin{proposition}\label{P:f(0)_f(a)}
    Let $f:M\rightarrow N$ be an homomorphism  of meadows. Then
    \begin{enumerate}
        \item $f(0_M)=0_{N}$.
        \item $f(\abf_M)=\abf_{N}$.
        \item $f(-1_M)=-f(1_M)$.
        \item If $f(x)=\abf_{N}$, then  $f(x+y)=f(x \cdot y)=\abf_{N}$, for all $x,y\in M$.
        \item The following are equivalent:
            \begin{enumerate}
                \item If $f(x)=1_{N}$ then $x = 1_{M}$.
                \item If $f(x)=0_{N}$ then $x = 0_{M}$.
            \end{enumerate}
        \item Condition 5 of Definition~\ref{D:Morphism} entails that $f(1_M)=1_N$ is equivalent to $f(0_M)=0_N$.
    \end{enumerate}
\end{proposition}
\begin{proof}

We start by noting that $1_M-1_M=0_M\cdot 1_M=0_M$,  by $(\PM_7)$.
    \begin{enumerate}
        \item We have $f(1_M)=f(1_M+0_M)=f(1_M)+f(0_M)$.
        Since $f(1_M)=1_N$ we have $1_N=1_N+f(0_M)$ and hence $f(0_M)=0_N$.
        \item First note that $\abf_M=1_M+0_M\cdot \abf_M$ and $\abf_M=0_M\cdot \abf_M$. Then  $f(\abf_M)=1_N+0_N\cdot f(\abf_M)$ and $f(\abf_M)=0_N\cdot f(\abf_M)$. By $(\M_{12})$ we have that $f(\abf_M)\inv =(1_N+0_N\cdot f(\abf_M))\inv = 1_N+0_N\cdot f(\abf_M) = f(\abf_M) $. Hence $f(\abf_M)=f(\abf_M)\inv = (0_N\cdot f(\abf_M))\inv = 0_N\inv \cdot f(\abf_M)\inv = \abf_{N}\cdot f(\abf_M)=\abf_{N}$.
        \item We have $0_N=f(1_M-1_M)=f(1_M)+f(-1_M)$, then $-f(1_M)=f(-1_M)+0_N\cdot f(1_M)=f(-1_M)+ 0_N \cdot 1_N=f(-1_M)$. 
        \item We have $f(x+y)=f(x)+f(y)=\abf_{N}+f(y)=\abf_{N}=\abf_{N} \cdot f(y)=f(x)\cdot f(y)=f(x\cdot y)$.
        \item Assume $(a)$. If $f(x)=0_N$ we have that $f(x+1_M)=f(x)+f(1_M)=f(x)+1_N=1_N$ and so, $x+1_M=1_M$ and therefore $x=0_M$. Hence $(b)$ holds. 
        Assume now $(b)$. If $f(x)=1_N$ we have that $f(x-1_M)=f(x)+f(-1_M)=f(x)-1_N=1_N-1_N=0_N$ and so, $x-1_M=0_M$ and therefore $x=1_M$. Hence $(a)$ holds.
        \item By part 1 it is enough to show the implication from right to left. If $f(0_M)=0_N$, then $0_N=f(0_M)=f(1_M-1_M)=f(1_M)+f(-1_M)=f(1_M)-f(1_M)=f(1_M)-1_N$ and hence $f(1_M)=1_N$.\qedhere
    \end{enumerate}
\end{proof}
In the following we shall drop the subscripts whenever there is no ambiguity.

It is now straightforward to check that the homomorphic image of a common meadow is again a common meadow. It is also straightforward to check that the Cartesian product of meadows is again a meadow.

We now give some examples of homomorphisms of meadows, taking advantage of Theorem \ref{T:DirectedLattice}.
\begin{exmp}
    Consider the following lattices

\[\begin{tikzcd}
	{R_0} &&& {T_0} \\
	{R_1} && {T_1} && {T_2} \\
	{\{\abf\}} &&& {T_3} \\
	&&& {\{\abf\}} 
	\arrow["\Id", from=1-4, to=2-3]
	\arrow["\Id", from=2-3, to=3-4]
	\arrow["\Id"', from=1-4, to=2-5]
	\arrow["\Id"', from=2-5, to=3-4]
	\arrow[from=3-4, to=4-4]
	\arrow[from=2-1, to=3-1]
	\arrow[hook,"i"', from=1-1, to=2-1]
\end{tikzcd}\]

\noindent where $R_0=\z$ and $R_1=T_0=T_1=T_2=T_3=\q$,  $i:R_0\rightarrow R_1$ is the inclusion homomorphism, and $\Id$ is the identity map. Denote by $M$ the common meadow defined by the lattice on the left and $N$ the common meadow defined by the lattice on the right. And define the map $f:M\rightarrow N$ in the following way: $x\in R_0=\z$ corresponds to the same integer in $T_0$, that is $f(x)=x\in T_0=\q$, and $x\in R_1$  is sent to $f(x)=x\in T_1=\q$. One can easily see that $f$ is an homomorphism of meadows whose image is the common meadow that corresponds to the lattice
\[\begin{tikzcd}
	\z \\
	\q \\
	{\{\abf\}}
	\arrow[from=2-1, to=3-1]
	\arrow[hook,"i"', from=1-1, to=2-1]
\end{tikzcd}\]
\end{exmp}

\begin{exmp}
    Let $R_0,R_1,R_2$ be unital commutative rings and $f_1,f_2$ be ring homomorphisms, such that they defined a meadow $M$. Consider the following lattice
    
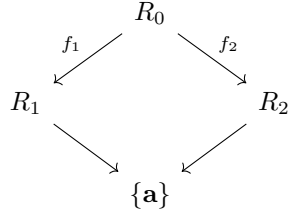
\[\begin{tikzcd}
	& {R_0} \\
	{R_1} && {R_2} \\
	& {\{\abf\}}
	\arrow[from=2-1, to=3-2]
	\arrow[from=2-3, to=3-2]
	\arrow["{f_{1}}"', from=1-2, to=2-1]
	\arrow["{f_{2}}", from=1-2, to=2-3]
\end{tikzcd}\]
which corresponds to a meadow $M$. From $M$ one can construct the lattice
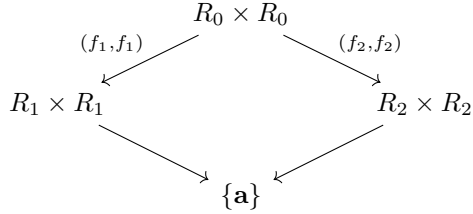
\[\begin{tikzcd}
	& {R_0\times R_0} \\
	{R_1\times R_1} && {R_2\times R_2} \\
	& {\{\abf\}}
	\arrow[from=2-1, to=3-2]
	\arrow[from=2-3, to=3-2]
	\arrow["{(f_{1},f_{1})}"', from=1-2, to=2-1]
	\arrow["{(f_{2},f_{2})}", from=1-2, to=2-3]
\end{tikzcd}\]
One easily sees that the lattice above defines a meadow, let us denote by $N$ such meadow. Note that $N$ is not the product of $M$ with $M$ because, for example, it does not contain $R_0\times R_1$. Define the map
$\pi_1:M\rightarrow N$ by $\pi_1(x,y)=x\in R_i$ for $(x,y)\in R_i\times R_i$. It is easy to see that $\pi_1$ defines an homomorphism of meadows that does not commute with the inverse (the inverse of $(0,1)$ is $\abf$, while $1$ is its own inverse).
\end{exmp}

\begin{exmp}
    Consider the lattice
    \[\begin{tikzcd}
	{R_0} \\
	{R_1} \\
	{\{\abf\}}
	\arrow[from=2-1, to=3-1]
	\arrow["f", from=1-1, to=2-1]
\end{tikzcd}\]
where $R_0$ and $R_1$ are unital commutative rings an $f$ is a ring homomorphism. Let $M$ be the common meadow corresponding to this lattice and let $N:=M\times M$. It is easy to see that $(0,0)\cdot N=0\cdot M\times 0\cdot M$. Then the lattice of $N$ corresponds to 

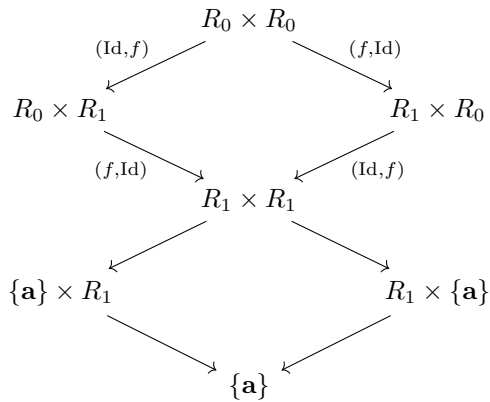
\[\begin{tikzcd}
	& {R_0\times R_0} \\
	{R_0\times R_1} && {R_1\times R_0} \\
	& {R_1\times R_1} \\
	{\{\abf\}\times R_1} && {R_1\times\{\abf\}} \\
	& {\{\abf\}}
	\arrow["{(f,\Id)}"', from=2-1, to=3-2]
	\arrow["{(\Id,f)}", from=2-3, to=3-2]
	\arrow["{(\Id,f)}"', from=1-2, to=2-1]
	\arrow["{(f,\Id)}", from=1-2, to=2-3]
	\arrow[from=3-2, to=4-1]
	\arrow[from=3-2, to=4-3]
	\arrow[from=4-1, to=5-2]
	\arrow[from=4-3, to=5-2]
\end{tikzcd}\]

One can easily see that the projection map  $\pi_1:N\rightarrow M$ is an homomorphism. Moreover, and similarly to what happens in ring theory, there is no injective homomorphism $h:M\rightarrow N$.
\end{exmp}


Every homomorphism of common meadows $f:M\rightarrow N$ can be restricted to an homomorphism of semigroups $f_0=f_{|0\cdot M}:0\cdot M\rightarrow 0\cdot N$ which preserves the partial order in $0\cdot M$. Moreover, for $y\in M_{0\cdot x}$ we have that $f(y)\cdot 0=f(y\cdot 0)=f(x\cdot 0)=f(x)\cdot 0$. That is, $f$ maps the elements in $M_{0\cdot x}$ to $M_{0\cdot f(x)}$. Hence, by restricting to $M_{0\cdot x}$ we have a ring homomorphism $f_{M_{0\cdot x}}=f_{0\cdot x}:M_{0\cdot x}\rightarrow M_{0\cdot f(x)}$. The following proposition summarizes this discussion.

\begin{proposition}\label{P:MorphismLattice}
         If $f:M\rightarrow N$ is an homomorphism of meadows then: 
     \begin{enumerate}
         \item The map $f_L:0\cdot M\rightarrow 0\cdot N$ is an homomorphism of lattices.
         \item The map $f_{0\cdot z}:M_{0\cdot z}\rightarrow M_{0\cdot f(z)}$ is an homomorphism of rings.
     \end{enumerate}
\end{proposition}

The following result gives a relation between homomorphisms of directed lattices and homomorphisms of common meadows, as to be expected by the relation between common meadows and lattices given by Theorem \ref{T:DirectedLattice}. 

\begin{proposition}\label{P:LatticeHom}
    Let $M$ and $N$ be common meadows, and let $f:M\rightarrow N$ be a function such that for all $x,y \in M$
    \begin{enumerate}
        \item $f(0\cdot x)=0\cdot f(x)$
        \item $f(0\cdot x\cdot y)=0\cdot f(x)\cdot f(y)$
        \item $f_{|M_{0\cdot z}}:M_{0\cdot z}\rightarrow M_{0\cdot f(z)}$ is a ring homomorphism for all $z\in M$.
        \item $f(x+0\cdot z)=f(x)+0\cdot f(z)$
    \end{enumerate}
    Then $f$ is an homomorphism of meadows.
\end{proposition}
\begin{proof}
    Let $x,y\in M$. Then
    \begin{align*}
        f(x+y)&=f(x+y+0\cdot x\cdot y) = f((x+0\cdot x\cdot y)+(y+0\cdot x\cdot y))\\
        &=f(x+0\cdot x\cdot y)+f(y+0\cdot x\cdot y)\\
        &= f(x)+0\cdot f(x\cdot y)+f(y)+0\cdot f(x\cdot y)\\
        &=f(x)+f(y)+0\cdot f(x)\cdot f(y)\\ &=f(x)+f(y)+0\cdot (f(x)+f(y))=f(x)+f(y).
    \end{align*}
From the fact that  $x\cdot y = x\cdot y + 0\cdot x\cdot y=(x+0\cdot x\cdot y)(y+0\cdot x\cdot y) $ we derive 
    \begin{align*}
        f(x\cdot y)&=f(((1+0\cdot x\cdot y)\cdot  x)\cdot ((1+0\cdot x\cdot y)\cdot y)) \\
        &=f((1+0\cdot x\cdot y)\cdot x) \cdot f((1+0\cdot x\cdot y)\cdot y)\\ 
        &= f(x+0\cdot x\cdot y) \cdot f(y+0\cdot x\cdot y)\\
        &=f(x)\cdot f(y)+0\cdot f(x)\cdot f(y)=f(x)\cdot f(y).
    \end{align*}
     The fact that $f(1)=1$ comes from the fact that $f_{|M_0}$ is a ring homomorphism. Then $f$ is an homomorphism of meadows.
\end{proof}

\subsection{Ideals}

Recall that if $R$ is a commutative  ring, a subset $I\subseteq R$ is said to be an \emph{ideal} of $R$ if $I$ is an abelian subgroup of $(R,+)$ and $x\cdot r\in I$, whenever $x\in I$ and $r\in R$. This notion can be adapted to the context of common meadows. 

\begin{definition}
Let $M$ be a meadow and $I\subseteq M$. We say that $I$ is an \emph{ideal} of $M$ if $-x,x+y,x\cdot r,0\in I$, whenever $x,y\in I$ and $r\in M$.  
\end{definition}

Let $M$ be a meadow and $I$ an ideal of $M$. It is an immediate consequence of the definition of ideal of a meadow that $\abf\in I$ and $0\cdot M\subseteq I$. 
In particular, for each $z\in M$, we have that $I\cap M_{0\cdot z}$ is an ideal of $M_{0\cdot z}$. 

Let $0\cdot z,0\cdot z'\in 0\cdot M$ be such that $0\cdot z\leq 0\cdot z'$. Then $x+0\cdot z'\in I\cap M_{0\cdot z}$, for all $x\in I\cap M_{0\cdot z'}$. 
Note that $I\cap M_{0\cdot z}=M_{0\cdot z}$ if and only if $1+0\cdot z\in I$, because $1+0\cdot z$ is the identity of the ring $M_{0\cdot z}$ and if $1+0\cdot z\in I$, for all $0\cdot z'\in 0\cdot M$, we have that $f_{0\cdot z',0\cdot z}(1+0\cdot z) = 1+0\cdot z+ z'=1+0\cdot z'\in I$, where $f_{0\cdot z',0\cdot z}$ is as in Proposition \ref{P:Transitionmaps}.  So we can define the quotient of a meadow by an ideal and the respective transition maps.

\begin{definition}\label{D:QuotientMeadow}
    Let $M$ be a meadow and $I$ an ideal of $M$. Let $$N:=\{0\cdot z\in 0\cdot M\mid I\cap M_{0\cdot z}=M_{0\cdot z}\}.$$ Then we can define the \emph{quotient} 
    $$M/I:=\left(\bigsqcup_{0\cdot z \in 0\cdot M\setminus N} M_{0\cdot z}/(I\cap M_{0\cdot z})\right)\sqcup \{\abf\}.$$ 
    For $0\cdot z,0\cdot z'\in M\setminus N$, such that $0\cdot z'\leq 0\cdot z$ we can define the \emph{transition maps}
     \begin{align*}
         \overline{f}_{0\cdot z',0\cdot z}:M_{0\cdot z}/(I\cap M_{0\cdot z})&\rightarrow M_{0\cdot z'}/(I\cap M_{0\cdot z'})\\
         x+I\cap M_{0\cdot z}&\mapsto f_{0\cdot z,0\cdot z'}(x)+I\cap M_{0\cdot z'},
     \end{align*}
     where $f_{0\cdot z',0\cdot z}$ is as in Proposition \ref{P:Transitionmaps}.
    \end{definition}
We now define operations on the quotient. 
    \begin{definition}\label{D:Operations}
     Let $M$ be a meadow, and $I$ an ideal of $M$. Let $0\cdot z,0\cdot z' \in 0\cdot M\setminus N$, with $N$ as in Definition \ref{D:QuotientMeadow}. The \emph{sum} and \emph{product} on the quotient $M/I$ are defined as follows. Given $x+(I\cap M_{0\cdot z})\in M_{0\cdot z}/(I\cap M_{0\cdot z})$ and $y+(I\cap M_{0\cdot z'})\in M_{0\cdot z'}/(I\cap M_{0\cdot z'})$, we define
    \[
    (x+I\cap M_{0\cdot z'})+(y+I\cap M_{0\cdot z})=
    \begin{cases}
        (x+y)+(I\cap M_{0\cdot z\cdot z'}),& \mbox{ if } 0\cdot z\cdot z'\notin N\\
        \abf,& \mbox{ if } 0\cdot z\cdot z'\in N,
        \end{cases}
     \]
     and 
     \[
     (x+(I\cap M_{0\cdot z'}))\cdot (y+(I\cap M_{0\cdot z}))=
        \begin{cases}
        (x\cdot y)+(I\cap M_{0\cdot z\cdot z'}),& \mbox{ if } 0\cdot z\cdot z'\notin N\\
        \abf,& \mbox{ if } 0\cdot z\cdot z'\in N.
    \end{cases}
    \]
\end{definition}

The following lemma summarizes some easy properties concerning quotients of rings.

\begin{lemma}\label{L:QuotientRing}
    Let $R$ be a (commutative) ring with unity and $I$ an ideal of $R$. Then the quotient $R/I=\{x+I\mid x\in R\}$ is a (commutative) ring, and there is a  surjective homomorphism $\pi:R\rightarrow R/I$ defined by $\pi(x)=x+I$. Moreover, for any ring homomorphism $\psi:R\rightarrow S$, with $f(I)=\{0\}$ there is a unique map $\widetilde{\psi}:R/I\rightarrow S$ such that $\widetilde{\psi}\circ\pi=\psi$.
\end{lemma}

We show  that the quotient of a meadow by an ideal is a pre-meadow with $\abf$, satisfying the usual properties of quotients of algebraic structures.
\begin{theorem}\label{T:QuotientUniversal}
    Let $M$ be a meadow and $I$ be an ideal of $M$ different from $M$. Then $M/I$ is a pre-meadow with $\abf$, and the mapping $\rho:M\rightarrow M/I$ defined by
    \[\rho(x)=
    \begin{cases}
        x+M_{0\cdot z}\cap I, &\text{ if } x\in M_{0\cdot z} \text{ and } 0\cdot z\notin N\\
        \abf,&\text{ otherwise}
    \end{cases}
    \]
    where $N$ is as in Definition \ref{D:QuotientMeadow}, is an homomorphism of meadows.
    
    Moreover, given an homomorphism of meadows $f:M\rightarrow N$ such that $0\cdot f(x)=f(x)$, for all $x\in I$, there exists a unique homomorphism of meadows $\widetilde{f}:M/I\rightarrow N$ such $\widetilde{f}\circ\rho=f$.
\end{theorem}
\begin{proof} 
We start by showing that $M/I$ is a pre-meadow with $\abf$. By Corollary~\ref{C:DirectedLattice} it is enough to show that there is a directed lattice $M/I=\Gamma_i$ but this is obvious from Definition~\ref{D:QuotientMeadow}.

Since $I\neq M$ we have that $0\notin N$ and then $\rho(1)=1+I\cap M_{0}$ is the identity of $M/I$. The fact that $\rho$ is an homomorphism of meadows then immediately follows from Definition \ref{D:Operations}.

    Let $0\cdot z\in 0\cdot M$ and $x\in I\cap M$. Then $f(x)\in N_{0\cdot f(z)}$ and $f(x)=0\cdot f(x)=0\cdot f(z)$. So $f(x)$ is mapped to the zero of $N_{0\cdot f(z)}$. Then, from Lemma \ref{L:QuotientRing} there exists a ring homomorphism 
     \begin{align*}
    \widetilde{f}_{0\cdot z}:M_{0\cdot z}/(I\cap M_{0\cdot z})\rightarrow N_{0\cdot f(z)}\\ \widetilde{f}_{0\cdot z}(x+I\cap M_{0\cdot z})=f(x).
    \end{align*}

    Let $x+I\cap M_{0\cdot z}$ and $y+I\cap M_{0\cdot z'}$, then 
    \begin{align*}
        \widetilde{f}_{0\cdot z\cdot z'}((x+I\cap M_{0\cdot z} )&+(y+I\cap M_{0\cdot z'}))\\
        &=\widetilde{f}_{0\cdot z\cdot z'}(x+y+I\cap M_{0\cdot z\cdot  z'})\\
        &=f(x+y)=f(x)+f(y)\\&=\widetilde{f}_{0\cdot z}(x+I\cap M_{0\cdot z})+\widetilde{f}_{0\cdot z}(y+I\cap M_{0\cdot z'}).
    \end{align*}
    In a similar way one verifies that $$\widetilde{f}_{0\cdot z}((x+I\cap M_{0\cdot z})\cdot(y+I\cap M_{0\cdot z'}))=\widetilde{f}_{0\cdot z\cdot z'}((x+I\cap M_{0\cdot z} )\cdot (y+I\cap M_{0\cdot z'})).$$

    Finally,  $\widetilde{f_0}(0+I\cap M_0)=f(0)=0$. Then we can glue the maps $\widetilde{f}_{0\cdot z}$, in order to define the homomorphism of meadows 
 $\widetilde{f}:M/I\rightarrow N$. 

    Finally, let $g:M/I\rightarrow N$ be such $g\circ\rho=f$, and $z\cdot 0 \in 0\cdot M$. Then, by restricting $g$ to $M_{0\cdot z}/(M_{0\cdot z})$, we have a ring homomorphism $g_{0\cdot z}$ such that $g_{0\cdot z}\circ \rho_{0\cdot z}=f_{0\cdot z}$. By Lemma \ref{L:QuotientRing} we have that $g_{0\cdot z}=\widetilde{f}_{0\cdot z}$, again by gluing we have $g=\widetilde{f}$ as we wanted.
\end{proof}

\begin{corollary}\label{C:QuotientUniversal}
    Let $M$ be a meadow and $I \neq M$ an ideal of $M$. Then $M/I$ is a meadow if and only if for all $x\in M$ the set:
    $$\overline{I}_x=\{0\cdot z\in 0\cdot M\mid x\cdot z -1 \in I\cap M_{0\cdot z}\}$$
    has a unique maximal element.
\end{corollary}

\begin{exmp}\label{E:QuotientColapse}
Consider the lattice
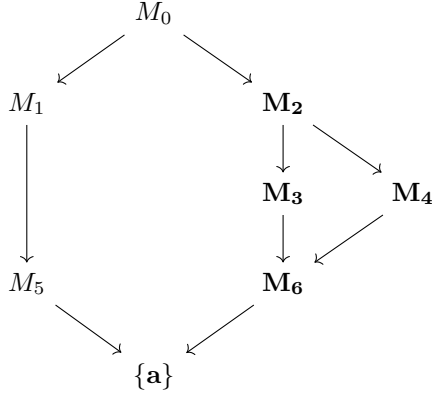
\[\begin{tikzcd}
	& {M_0} \\
	{M_1} && {\mathbf{M_2}} \\
	&& {\mathbf{M_3}} & {\mathbf{M_4}} \\
	{M_5} && {\mathbf{M_6}} \\
	& {\{\abf\}}
	\arrow[from=1-2, to=2-1]
	\arrow[from=1-2, to=2-3]
	\arrow[from=2-3, to=3-3]
	\arrow[from=2-3, to=3-4]
	\arrow[from=3-3, to=4-3]
	\arrow[from=3-4, to=4-3]
	\arrow[from=2-1, to=4-1]
	\arrow[from=4-1, to=5-2]
	\arrow[from=4-3, to=5-2]
\end{tikzcd}\]

Let $I:=\mathbf{M_2}\sqcup \mathbf{M_3}\sqcup \mathbf{M_4}\sqcup \mathbf{M_6}\sqcup \{\abf\}$, one easily sees that $I$ obeys the condition of Corollary~\ref{C:QuotientUniversal}. Then, the quotient $M/I$ corresponds to the lattice
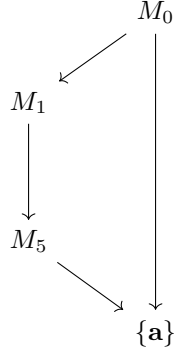
\[\begin{tikzcd}
	& {M_0} \\
	{M_1} \\
	\\
	{M_5} \\
	& {\{\abf\}}
	\arrow[from=1-2, to=2-1]
	\arrow[from=2-1, to=4-1]
	\arrow[from=4-1, to=5-2]
	\arrow[from=1-2, to=5-2]
\end{tikzcd}\]

That is, the ideal $I$ is identified with $\abf$.

\end{exmp}

\subsection{Kernels and Isomorphisms}

Since common meadows are directed lattices of rings we can define the kernel of an homomorphism related either with the ring structure or with the lattice structure. Let $M$ and $N$ be common meadows, $f:M\rightarrow N$ be an homomorphism of common meadows, and $0\cdot z\in 0\cdot N$. We  define 
\begin{enumerate}
    \item $\Ker^R(f):=\{x\in M\mid f(x)=0\cdot f(x)\}$
    \item $\Ker^{0\cdot z}(f):=\{x\in M\mid f(x)=0\cdot z\}$.   
\end{enumerate}

If $z=0$, then $0\cdot z = 0$, and if $z=\abf$, then $0\cdot z = \abf$. In these cases we write simply $\Ker^{0}(f)$ and $\Ker^{\abf}(f)$.

Recall that the kernel of a ring homomorphism is the pre-image of the zero of the ring. However, in a common meadow $M$ we have "generalized zeros", that is, for all $z\in N$, the element $0\cdot z$ is the zero of the ring $N_{0\cdot z}$. A consequence of this is that the kernel  $\Ker^{0\cdot z}(f)$ may be empty, if $0\cdot z$ is not in the image of $f$. We illustrate this possibility in the following example.

\begin{exmp}
     Consider the following lattices 
    \[\begin{tikzcd}
	&& {S_0} \\
	{R_0} && {S_1} \\
	{\{\abf\}} && {\{\abf\}}
	\arrow[from=2-1, to=3-1]
	\arrow[from=2-3, to=3-3]
	\arrow["\Id"', from=1-3, to=2-3]
\end{tikzcd}\]
where $S_0=S_1=R_0=\z$, and let $M=R_0\sqcup\{\abf\}$ and $N=S_0\sqcup S_1\sqcup\{\abf\}$ be the common meadows related with the lattices above. Define the homomorphism of meadows $f:M\rightarrow N$ by taking $x\in R_0$ and mapping it to $f(x)=x\in S_0$. One can easily see that $\Ker^0(f)=\{0\}$. However, by taking $z\in S_1$ we have that $\Ker^{0\cdot z}(f)$ is empty because $S_1$ does not intersect the image of $f$.
\end{exmp}

Recall from Proposition \ref{P:MorphismLattice} that given an homomorphism of meadows $f:M\rightarrow N$ there are associated ring homomorphisms  $f_{0\cdot z}:M_{0\cdot z}\rightarrow N_{0\cdot f(z)}$ with $z\in M$. One can easily see that $\Ker^R(f)\cap M_{0\cdot z}$ is the kernel of the ring homomorphism $f_{0\cdot z}$. This means that $\Ker^R(f)$ measures the injectivity of the homomorphisms $f_{0\cdot z}$.

\begin{proposition}\label{P:IsomorphTheoremLocal}
    Let $f:M \to N$ be an homomorphism of meadows. Then 
    \begin{enumerate}
        \item $0\cdot M\in \Ker^R(f)$
        \item $\Ker^R(f)$ is an ideal of $M$
        \item $\Ker^R(f)=\bigsqcup_{z\in M}(\Ker^R(f))_{0\cdot z}$, where $(\Ker^R(f))_{0\cdot z}=\Ker^R(f)\cap M_{0\cdot z}$ is an ideal of  $M_{0\cdot z}$. 
        \item $0\cdot M= \Ker^R(f)$ if and only if for all $z\in M$ the ring homomorphism $f_{0\cdot z}$ is injective
        \item There is an homomorphism of pre-meadows with $\abf$, defined by
        \begin{align*}
            \overline{f}:M/\Ker^R(f)\rightarrow N\\
            \overline{f}(x+\Ker^R(f))=f(x),
        \end{align*}  
        such that, for all $\overline{z}\in M/\Ker^R(f)$, the ring homomorphism $\overline{f}_{0\cdot \overline{z}}$ is injective.
    \end{enumerate}
\end{proposition}
\begin{proof}
\begin{enumerate}
    \item Let $0\cdot z\in 0\cdot M$. Then $f(0\cdot z)=f(0\cdot 0\cdot z)=0\cdot f( 0\cdot z)$, and so $0\cdot z\in \Ker^R(f)$.
    \item Let $x,y\in \Ker^R(f)$ and $b\in M$. Then 
    \begin{itemize}
        \item $f(-x)=-f(x)=-(0\cdot f(x))=0\cdot f(x)$
        \item $f(x+y)=f(x)+f(y)=0\cdot f(x)+0\cdot f(y)=0\cdot (f(x)+f(y))=0\cdot(f(x+y))$
        \item $f(x\cdot b)=f(x)\cdot f(b)=0\cdot f(x)\cdot f(b)=0\cdot f(x\cdot b)$
        \item $f(0)=f(0\cdot 0)=0\cdot f(0)$.
    \end{itemize}
    Hence $\Ker^R(f)$ is an ideal of $M$.
    \item The fact that $\Ker^R(f)=\bigsqcup_{z\in M}(\Ker^R(f))_{0\cdot z}$ follows from the fact that $M=\bigsqcup_{z\in M}M_{0\cdot z}$. Since $M_0\cdot z$ is a ring and $\Ker^R(f)$ is an ideal then $(\Ker^R(f))_{0\cdot z}=\Ker^R(f)\cap M_{0\cdot z}$ must be an ideal of $M_{0\cdot z}$.
    \item Note that the ring homomorphism $f_{0\cdot z}:M_{0\cdot z}\rightarrow N_{0\cdot f(z)}$ is injective if and only if $\Ker(f_{0\cdot z})=\Ker^R(f)\cap M_{0\cdot z}=\{0\cdot z\}$. Then $f_{0\cdot z}$ is injective, for all $z\in M$, if and only if $0\cdot M=\Ker^R(f)$.

    \item Since $\Ker^R(f)$ is an ideal of $M$ and $0\cdot f(x)=f(x)$ for all $x\in \Ker^R(f)$, by Theorem \ref{T:QuotientUniversal} there exists a (unique) homomorphism of meadows $\overline{f}:M/\Ker^R(f)\rightarrow N$ such that $\overline{f}\circ\rho =f$, where $\rho:M\rightarrow M/\Ker^R(f)$. Let $z\in M$, and $x\in M_{0\cdot z}$, such that $\overline{f}(x+\Ker^R(f)\cap M_{0\cdot z})= 0\cdot(\overline{f}(x+\Ker^R(f)\cap M_{0\cdot z}) ) $. Then 
    $$f(x)=\overline{f}(x+\Ker^R(f)\cap M_{0\cdot z})= 0\cdot f(x) $$
    and so $x\in \Ker^R(f)\cap M_{0\cdot z}$. Hence the map $f_{0\cdot \overline{z}}$ is injective. \qedhere
\end{enumerate}

\end{proof}
We note that the homomorphism $\overline{f}$ in Proposition \ref{P:IsomorphTheoremLocal} may fail to be injective as seen in the following example.

\begin{exmp}
    Take the following lattices
    \[\begin{tikzcd}
	&& {S_0} \\
	&& {S_1} && {R_0} \\
	{} && {\{\abf\}} && {\{\abf\}}
	\arrow[from=2-3, to=3-3]
	\arrow["\Id"', from=1-3, to=2-3]
	\arrow[from=2-5, to=3-5]
\end{tikzcd}\]

where $S_0=S_1=R_0=\z$, and let $M=S_0\sqcup S_1\sqcup\{\abf\}$ and $N=R_0\sqcup\{\abf\}$ be the common meadows related with the lattices above. Then define the homomorphism of common meadow  $f:M\rightarrow N$ in the following way for all $x\in M\setminus\{\abf\}$ we send it to $f(x)=x\in R_0$.

Then it is easy to see that the ring homomorphism $f_0:S_0/(\Ker^R(f)\cap S_0)\rightarrow R_0$ and $f_{0\cdot z}:S_0/(\Ker^R(f)\cap S_1\rightarrow R_0)$, with $z\in S_1$, are both injective, however $f$ is not injective.
\end{exmp}

From Proposition \ref{P:MorphismLattice},  given an homomorphism of meadows $f:M\rightarrow N$ there are an associated homomorphism of lattices $f_L:0\cdot M\rightarrow 0\cdot N$, and a ring homomorphism $f_{0\cdot z}:M_{0\cdot z}\rightarrow N_{0\cdot f(x)}$. Recall that the previous example entails that a common meadow may fail to be injective even if all ring homomorphisms $f_{0\cdot z}$ are injective. The following proposition gives a characterization of injectivity.

\begin{theorem}\label{P:Inject}
    Let $f:M\rightarrow N$  be  an homomorphism of meadows, then $f$ is injective if and only if the maps $f_L$ and $f_{0\cdot z}$, introduced in Proposition \ref{P:MorphismLattice}, are injective homomorphisms for all $z\in M$.
\end{theorem}
\begin{proof}
    If $f$ is injective then it is clear that $f_L$ and $f_{0\cdot z}$ are injective homomorphism for all $z\in M$.

    Now suppose that $f_L$ and $f_{0\cdot z}$ are injective homomorphisms for all $z\in M$. Let  $x,y\in M$ be such that $f(x)=f(y)$. Then $0\cdot f(x)=0\cdot f(y)$ or, equivalently, $f(0\cdot x)=f(0\cdot y)$. Since $f_L$ is injective we have that $0\cdot x=0\cdot y$. Finally, since $f(x)=f_{0\cdot x}(x)$ and $f(x)=f_{0\cdot y}(y)=f_{0\cdot x}(y)$ we have that $f_{0\cdot x}(x)=f_{0\cdot x}(y)$ and from the injectivity of $f_{0\cdot x}$ it follows that $x=y$.
\end{proof}

From Theorem \ref{P:Inject} one easily obtains the following result.

\begin{corollary}
    Let $f:M \to N$ be an homomorphism of meadows. Then $f$ is injective if and only if for all $0\cdot z\in 0\cdot N$ the kernel $\Ker^{0\cdot z}(f)$ has at most one element.    
\end{corollary}

\begin{definition}
    An \emph{isomorphism of meadows} is a bijective homomorphism of meadows. 
\end{definition}

Theorem \ref{P:Inject} also entails the following version of the first isomorphism theorem.

\begin{corollary}\label{T:IsomorphismTheorem}
    Let $f:M \to N$ be a surjective homomorphism of meadows, such that $f_L:0\cdot M\rightarrow 0 \cdot N$ is an isomorphism of lattices. Then $M/\Ker^R(f)$ is a meadow and  $\overline{f}:M/\Ker^R(f)\rightarrow N$, defined by $\overline{f}(x+\Ker^R(f))=f(x)$  is an isomorphism of meadows.
\end{corollary}

Recall that for isomorphisms of monoids the notion of kernel is replaced by the notion of congruence. In our setting, the same is needed in order to be able to show a proper version of the first isomorphism theorem.

\begin{definition}\label{D:Relation}
    Let $f:M\rightarrow N$, and let $x,y\in M$ then we say that $x\sim_f y$

    \begin{equation}\label{Eq:Equiv_rel}
\forall x,y\in M \left(x\sim_f y \iff f(x)=f(y)\right). 
\end{equation}
\end{definition}

Since a morphism of common meadows is also a morphism of the monoids $(M,+)$ and $(M,\cdot)$, by \cite[Theorem 1.5.2]{howie1995fundamentals} we have that the relation in Definition \ref{D:Relation} is an equivalence relation.  From \cite[Theorem 1.5.2]{howie1995fundamentals} we have that $M/\sim_f$ is a pre-meadow with $\abf$. Also,  there exists an isomorphism of common meadows $\overline{f}:M/\!\sim_f \ \rightarrow f(M)$ defined by $\overline{f}([x]_{\sim_f})=f(x)$, for all $x\in M$. In the case $f$ is surjective we have that $M/\!\sim_f$ is indeed a meadow. Hence we have the following result.

\begin{proposition}\label{P:LatticeQuotient}
    Let $f:M\rightarrow N$ be an surjective homomorphism of meadows. Then $M/\sim_f$ is a meadow, where $\sim_f$ is the equivalence relation defined by \eqref{Eq:Equiv_rel}.

    Moreover, the map $\overline{f}:M/\!\sim_f \ \rightarrow N$ defined by $\overline{f}([x]_{\sim_f})=f(x)$, for all $x\in M$ is an isomorphism of meadows.
\end{proposition}

Note that the inverse of the equivalence class $[x]_{\sim_f}$ may fail to be $[x\inv]_{\sim_f}$ as illustrated in the example below.
\begin{exmp}
    Let $M=(\z_2\times\z_2)\sqcup\{\abf\}$ and $N=\z_2\sqcup\{\abf\}$. One easily sees that for all $x\in M\setminus\{(1,1)\}$ we have $x\inv= \abf$. Define the map $f:M\rightarrow N$ by $f(a,b)=a\in N$, and $f(\abf)=\abf$. Then  $M/\sim_f=\{ \{(0,0),(0,1)\},\{(1,1),(1,0)\},\{\abf\} \}$ and $\overline{f}(\{(1,1),(1,0)\})=1$. But $(1,0)\inv =\abf$ while $1\inv =1$.
\end{exmp}

Observe that in the case where the map $f_L$ is injective, the quotient in Proposition \ref{P:LatticeQuotient} coincides with the quotient defined in Corollary \ref{T:IsomorphismTheorem}.

Let $M$ be a meadow. By Proposition~\ref{P:Assembly}, the structure $(M,+)$ is an assembly. We may then consider the set of all assembly homomorphisms (see \cite{DDT} for the definition) from $(M,+)$ to itself. We denote this set by $\mathrm{End}(M)$. The proof of the following theorem is similar to the proof of Proposition \ref{P:Assembly} so it shall be omitted here.
\begin{theorem}
    Let $M$ be a meadow. Then  $(\mathrm{End}(M),+)$ is an assembly. 
\end{theorem}

If $M$  is a structure satisfying all the conditions in Definition~\ref{D:PreMeadow} except $(\PM_6)$, we say that $M$ is a \emph{non-commutative pre-meadow}. Non-commutative pre-meadows with $\abf$ are defined in an analogous way.

\begin{theorem}
    Let $M$ be a meadow. Then $(\mathrm{End}(M),+,\circ)$ is a non-comutative pre-meadow with $\abf$. Additionally, the map
    \begin{align*}
        M&\rightarrow \mathrm{End}(M)\\
        d&\mapsto \phi_d,
    \end{align*}
    where $\phi_d(x)=dx$  is injective and its image is a common meadow isomorphic to $M$.
\end{theorem}
\begin{proof}
    The fact that $(\mathrm{End}(M),+,\circ)$ is a non-comutative pre-meadow with $\abf$ easily follows from the definition and the fact that $M$ is a common meadow.

    Additionally, if we take $d_1,d_2\in M$ such that $\phi_{d_1}=\phi_{d_2}$, then $\phi_{d_1}(1)=\phi_{d_2}(1)$, that is $d_1=d_2$. Hence the map $\phi$ is injective. Note that the map $\phi$ has the following properties
    \begin{itemize}
        \item $\phi(d_1+d_2)=\phi(d_1)+\phi(d_2)$
        \item $\phi(d_1 \cdot d_2)=\phi(d_1)\cdot \phi(d_2)$
        \item $\phi(1)=\Id_M$
        \item $\phi(0)=0_{\mathrm{End}(M)}$
        \item $\phi(\abf)=\abf_{\mathrm{End}(M)}$,
    \end{itemize}
    from which one concludes that the image of $\phi$ is a meadow isomorphic with $M$.
\end{proof}

In Example \ref{E:QuotientColapse} we have an homomorphism $\rho:M\rightarrow M/I$ such that $\Ker^{\abf}(\rho)$ is equal to $I$. That is, given an homomorphism $f:M\rightarrow N$, the set $\Ker^{\abf}(f)$ is an ideal of $M$ which measures how much the common meadow $M$ collapses to $\abf$ when forming the quotient $M/\Ker^{\abf}(f)$.

\section{Alternative meadows}\label{S:Alternative}

In \cite{Bergstra2015}, the authors considered the possibility of adding other axioms to common meadows. Here we consider those axioms (see  Figure~\ref{tab:other_axioms}), present some examples and give a different proof for \cite[Proposition 3.1.1]{Bergstra2015} profiting from the fact that we may now see common meadows as lattices of rings. It is not difficult to see that (over the axioms of common meadows)  $\NVL+\AVL=\CIL$. Examples \ref{E:NVL} and  \ref{E:AVL} illustrate this equality.

\begin{figure}[ht]
\centering
\begin{tabular}{l l c r}
\toprule
\\[-2mm]
$x \neq \abf \to 0 \cdot x = 0$ & Normal Value Law &($\NVL$)\\
$x~\inv = \abf \to 0 \cdot x = x$ & Additional Value Law &($\AVL$)\\
$x \neq 0 \wedge x \neq \abf \to x \cdot x\inv = 1$ & Common Inverse Law &($\CIL$)\\
\toprule
\end{tabular}
\caption{Additional axioms for common meadows}
\label{tab:other_axioms}
\end{figure}

The following examples illustrate how one can interpret $\NVL, \AVL$ and $\CIL$ in terms of lattices. 
\begin{exmp}\label{E:NVL}
Let $R$ be a ring and $M=R\sqcup\{\abf\}$ be a meadow defined by the following lattice
\[\begin{tikzcd}[ampersand replacement=\&]
	R\\
	{\{\abf\}}
	\arrow[from=1-1, to=2-1]
\end{tikzcd}\]
The element $0\in M$ behaves as in a ring, that is $0\cdot x = 0$, for all $x\in M\setminus\{\abf\}$. Hence $\NVL$ is satisfied. In fact, $\NVL$ postulates that the associated lattice must be similar to the lattice in this example.
\end{exmp}

\begin{exmp}\label{E:AVL}
    Consider the following lattices
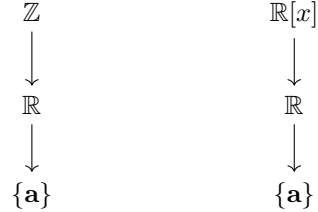
\[\begin{tikzcd}
	{\mathbb{Z}} &&&  {\mathbb{R}[x]}\\
	{\mathbb{R}} &&& {\mathbb{R}} \\
	{\{\abf\}} &&& {\{\abf\}}
	\arrow[from=1-1, to=2-1]
	\arrow[from=2-1, to=3-1]
	\arrow[from=1-4, to=2-4]
	\arrow[from=2-4, to=3-4]
\end{tikzcd}\]

Let $M$ be the meadow related with the lattice on the left and $N$ the meadow related with the lattice on the right. We have that $M$ satisfies $\AVL$, while $N$ does not. Indeed, the only elements of $M$ whose inverse is $\abf$ are the elements $x\in M$ such that $0\cdot x = x$, that is, the zeros of the rings $M_{0\cdot x}$.

As for $N$, the polynomials in $\r[x]$ that have no constant term are elements $x\in N$ such that $0\cdot x \neq x$, but have inverse equal to $\abf$.
\end{exmp}

\begin{exmp}
    Suppose that $M$ is a meadow satisfying $\CIL$ and $0\cdot M$ is the lattice

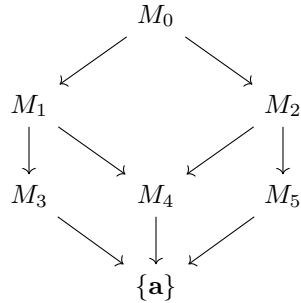
\[\begin{tikzcd}
	& {M_0} \\
	{M_1} && {M_2} \\
	{M_3} & {M_4} & {M_5} \\
	& {\{\abf\}}
	\arrow[from=1-2, to=2-1]
	\arrow[from=1-2, to=2-3]
	\arrow[from=2-1, to=3-1]
	\arrow[from=2-1, to=3-2]
	\arrow[from=2-3, to=3-2]
	\arrow[from=2-3, to=3-3]
	\arrow[from=3-2, to=4-2]
	\arrow[from=3-1, to=4-2]
	\arrow[from=3-3, to=4-2]
\end{tikzcd}\]

    For $x\in M_1$, we have that the inverse of $x$ belongs to $M_x$, where $M_x$ is either $M_3$, $M_4$, or $\{\abf\}$. The product $x\cdot x\inv$ must also be in $M_x$. Since $M$ satisfies $\CIL$ we have that if $x\neq 0$ and $x\neq \abf$ then $x\cdot x\inv =1\in M_0$. Then, in fact there is no $M_1$ and  no $M_2$, and the lattice collapses to
    \[\begin{tikzcd}
	{M_0} \\
	{\{\abf\}}
	\arrow[from=1-1, to=2-1]
\end{tikzcd}\]

    Additionally, for all $x\in M_0$ such that $x\neq 0$ and $x\neq \abf$, by the same reasoning $x\inv \in M_0$. This implies that $M_0$ must be a field.
\end{exmp}

\begin{proposition}\label{P:NVL}
    Let $M$ be a meadow. Then 
    \begin{enumerate}
        \item  $M$ satisfies $\NVL$ if and only if $0\cdot M=\{\abf,0\}$.
        \item If $M$ satisfies $\AVL$, then $M_z$ is a field for every minimal element $z\in 0\cdot M\setminus \{\abf\}$.
        \item  $M$ satisfies both $\NVL$ and $\AVL$ if and only if $M=\mathbb{F}\sqcup \{\abf\}$, where $\mathbb{F}$ is a field. 
        \item $M$ satisfies $\CIL$ if and only if $M=M_0\sqcup\{\abf\}$ and $M_0$ is a field.
    \end{enumerate}
\end{proposition}
\begin{proof}
\begin{enumerate}
    \item If $M$ satisfies $\NVL$, then clearly $0\cdot M = \{\abf,0\}$. 
    
    Conversely, suppose that $M$ is a meadow such that $0\cdot M = \{\abf,0\}$, and let $x\neq \abf$. Then $0\cdot x\neq \abf$, by Proposition \ref{P:Identities} $(8)$. Hence  $0\cdot x=0$, i.e.\ $M$ satisfies $\NVL$.
    \item    Let $x\in M_z$ be non-invertible in $M_z$. Since $z$ is minimal in $0\cdot M\setminus \{\abf\}$, it must be the case that the inverse of $x$ is $\abf$. Now, since $\AVL$ holds, we have $0\cdot x = x$, which means that $0\cdot x$ is the zero of the ring $M_z$. Hence $M_z$ is a field.
    \item The proof follows immediately from Part 1 and Part 2.
    \item Let $x\in M\setminus\{0,\abf\}$. One can easily see that $0\cdot x \cdot x\inv\leq 0\cdot x$. Since $M$ satisfies $\CIL$ we have 
$$0 =0\cdot 1= 0\cdot x \cdot x\inv\leq 0\cdot x.$$

Now, since $0$ is the maximum of $0\cdot M$ we must have $0\cdot x = 0$, that is $x\in M_0$. Then $M=M_0\sqcup \{\abf\}$. Also, $\CIL$ entails that for $x\in M_0\setminus \{0\}$,  the inverse of $x$ must be in $M_0$, and so $M_0$ is a field. \qedhere
\end{enumerate}
\end{proof}

From Part 1 of Proposition \ref{P:NVL} we have that if $M$ is a common meadow satisfying $\NVL$, then $M=M_0 \sqcup\{\abf\}$. And the inverse of the elements of $M_0$ that have no inverse in $M_0$ is $\abf$. 

\begin{proposition}\label{P:IdealNVL}
    Let $M$ be a meadow. The set
    $$\Rad(M):=\{x\in M\mid 0 \cdot x\neq 0\}\sqcup \{0\}$$
    is an ideal of $M$ and the quotient $M/\Rad(M)$ is a meadow satisfying $\NVL$. Moreover, $Rad(M)$ is the smallest ideal with that property.
\end{proposition}
\begin{proof}
  The fact that $\Rad(M)$ is an ideal follows from Proposition \ref{P:Identities}. By Theorem 2.5 we have that $\Rad(M)=\bigsqcup_{0\cdot z\in 0\cdot M\setminus\{0\}}M_{0\cdot z}\sqcup\{0\}$. Then $M/\Rad(M)=M_0/\{0\}\sqcup \{\abf
\}$ is a meadow, and so by Proposition \ref{P:NVL} we have that $M/\Rad(M)$ satisfies $\NVL$. 

Now take an ideal $I$ of $M$ different from $M$ such that $M/I$ satisfies $\NVL$. Note that we must have $I\cap M_0\neq M_0$. Since $M/I$ satisfies $\NVL$ we must have $M/I=I\cap M_0\sqcup \{\abf\}$, and so $R(M)\subseteq I$.
\end{proof}

Recall that given a unital commutative ring $R$ and $I$ an ideal of $R$ then the quotient $R/I$ is a field if and only if $I$ is a maximal ideal in $R$. A similar result holds for ideals of meadows.
\begin{proposition}
    Let $M$ be a meadow and $I$ an ideal of $M$. Then $M/I$ satisfies $\CIL$ if and only if $I$ is a maximal ideal of $M$.
\end{proposition}
\begin{proof}
    Let $I$ be a maximal ideal of $M$. Then $\Rad(M)\subseteq I$, and $I\cap M_0$ is a maximal ideal of $M_0$ from which we conclude that  $M_0/I\cap M_0$ is a field. Then $M/I=(M_0/I\cap M_0)\sqcup\{\abf\}$ is a meadow that satisfies $\CIL$, by Proposition \ref{P:NVL}.

    Suppose now that $M/I$ satisfies  $\CIL$. Then $M/I=(M_0/I\cap M_0)\sqcup\{\abf\}$, where $M_0/I\cap M_0$ is a field. So, $I\cap M_0$ is a maximal ideal of $M_0$. Since $0\cdot M/I$ only has two elements, we have that $R(M)\subseteq I$ and therefore $I$ is a maximal ideal of $M$.
\end{proof}

\section{The category of common meadows}\label{S:Categories}
In this section we briefly consider common meadows from a categorical perspective in order to showcase some possibilities of further research. For basic notions on category theory we refer to \cite{maclane:71}.
\begin{definition}
    The category of meadows $\Md$ is defined as follows:
\begin{itemize}
    \item The \emph{objects} of $\Md$ are common meadows.
    \item If $M,N \in \Md$, then the \emph{morphisms} from $M$ to $N$ are the elements of $\HomM$, the set of all homomorphisms of meadows from $M$ to $N$.
\end{itemize}
\end{definition}

We shall denote by $\rm{CRing}$ the category of commutative rings with unity, whose objects are commutative rings with unity and morphisms are ring homomorphisms. Recall the if $M$ is a meadow, $M_0$ denotes the set $\{x\in M\mid 0\cdot x = 0  \}$.
\begin{proposition}\label{P:Functors}
    The correspondence that sends a meadow $M$ to the ring $M_0$ and each common meadow homomorphism $f:M\rightarrow N$ to the restriction $f_0:M_0\rightarrow N_0$ defines a functor $R:\Md \rightarrow \rm{CRing}$.
    Conversely, the correspondence that sends a unital commutative ring $R$ to the common meadow $M(R)=R\sqcup\{\abf\}$, and each ring homomorphism $f:R\rightarrow S$ to an  homorphism of meadows $f':M\rightarrow N$ defines a functor $M:\rm{CRing}\rightarrow\Md$.
\end{proposition}

        \begin{proof}
            The correspondence that sends a meadow $M$ to the ring $M_0$, defines a correspondence between objects of the category $\Md$ and the category of rings.
            
            Let $f:M\rightarrow N$ be an homomorphism of meadows and let $x\in M_0$. Then $f(x)\cdot 0 = f(x\cdot 0) = f(0)=0$. That is, the homomorphism of meadows defines a ring homomorphism $f_0:M_0\rightarrow N_0$ by restriction. It is easy to see that if $g:N\rightarrow N'$ is another meadow homomorphism then $(g\circ f)_0=g_0\circ f_0$. 

             The fact that $M=R\sqcup\{\abf\}$ is a common meadow was seen in Example \ref{E:RingHom}. 
    
    Given a ring homomorphism $f:R\rightarrow S$, the map defined by $f'(x)=f(x)$ if $x\in R$ and $f'(\abf)=\abf$ is clearly a ring homomorphism. Additionally, if $g:S\rightarrow T$ is another ring homomorphism one can easily see that $(g\circ f)'$ is equal to $g'\circ f'$.
	\end{proof}

Recall that given functors $F:\mathcal{C}\rightarrow\mathcal{D}$ and $G:\mathcal{D}\rightarrow\mathcal{C}$ one says that $F$ is \emph{right adjoint} to $G$ if there exists a bijection $\mathrm{Hom}_\mathcal{D}(X,F(X))\rightarrow \mathrm{Hom}_\mathcal{C}(G(X),Y)$ for all objects $X\in \mathcal{C}$ and $Y\in \mathcal{D}$. The right adjoint functors are extremely well behaved, in particular they commute with limits, and preserve right exact sequences.
We prove that the functors $M,R$ defined in Proposition \ref{P:Functors} share this relation.
    
 \begin{theorem}
     The functor $M$ is the right adjoint of the functor $R$.
 \end{theorem}
 \begin{proof}
     Let $M$ be a common meadow, $R$ a unital commutative ring, and $f:R\rightarrow M_0$ a ring homomorphism. Let $M(R)=R\sqcup\{\abf\}$. Then we can define an homomorphism of meadows $f':M(R)\rightarrow M$ in the following way: we send each $x\in R\subseteq N$ to $f'(x)=f(x)\in M_0\subseteq M$, and $f'(\abf)=\abf$. It is straightforward to see that this defines an  homomorphism of meadows. From the construction we see that the correspondence $f\mapsto f'$ is injective. Now let $g:M(R)\rightarrow M$ be an homomorphism of meadows, we have that $g_0:N_0\rightarrow M_0$ is a ring homomorphism, and since $N_0=R$. We have that $(g_0)'=f$. Hence the correspondence $f\mapsto f'$ is a bijection.
 \end{proof}

\begin{definition}
    The category $\mathrm{LatRing}$ is the category  whose objects are directed lattices of rings $\Gamma=(L,(R_i,f_{i,j})_{i,j\in L})$  and the morphism $\varphi:\Gamma'\rightarrow\Gamma'$, where $\Gamma=(L,(R_i,f_{i,j})_{i,j\in L})$ and $\Gamma'=(L',(R'_i,f'_{i,j})_{i,j\in L'})$ is a lattice homomorphism of the lattices of $\varphi_L:L\rightarrow L'$ together with ring homomorphisms for each $i\in L$ $\varphi_i:R_i\rightarrow R_{\varphi_L(i)}$ such that $ \varphi_j\circ f_{i,j}=f_{\varphi_L(i),\varphi_L(j)}\circ \varphi_i$ whenever $j\leq i$.
\end{definition}
\begin{definition}
    The category $\mathrm{LatRing}^{M}$ is the subcategory of $\mathrm{LatRing}$ such that for all object $\Gamma=(L,(R_i,f_{i,j})_{i,j\in L})\in\mathrm{LatRing}^{M} $ we have for all $i\in L$, and $x\in R_i$ the set: 
    $$J_x=\{j\in L\mid f_{j,i}(x)\in (R_i)^\times\}$$
\end{definition}
Combining Theorem \ref{T:DirectedLattice} and Propositions \ref{P:MorphismLattice} and \ref{P:LatticeHom} we obtain the following theorem.
        
        \begin{theorem}
           The category  $\mathrm{LatRing}^M$ is equivalent to the category $\Md$. 
        \end{theorem}
 The category $\Md$ shares the initial object with the category $\mathrm{CRing}$.
        \begin{proposition}
            The category $\Md$ has an initial object which is $\z\sqcup \{\abf\}$.
        \end{proposition}
        \begin{proof}
            Note that any homomorphism from $M=\z\sqcup \{\abf\}$ to a common meadow $N$ is defined by $f_0:M_0\rightarrow N_0$, and such homomorphism is uniquely defined by the value $f(1)$, which is always equal to $1$. Then $M=\z\sqcup \{\abf\}$ is an initial object in the category $\Md$.
        \end{proof}

        Note that in the subcategory $\Md\inv$, where the objects are the same as $\Md$ but the morphisms are the meadow homomorphisms that commute with the inverse, there is no  initial object.

        Given $R$, a fixed commutative ring with unit, we can also study $\Md^R$, the full subcategory of $\Md$ such that for all $M\in \Md^R$ we have that $M_0$ is isomorphic to $R$.

        \begin{proposition}
            Let $M\in\Md^R$ be a common meadow, where $R$ is a commutative ring with unity. Then $M_{0\cdot z}$ is an $R$-module, for all $z\in M$, and $M$ is a directed lattice of $R$-modules.
        \end{proposition}
        \begin{proof}
            Let $x\in M_{0\cdot z}$ and $r\in M_0$. We have that $r\cdot x \in M_{0\cdot z}$, and from the fact that $M$ is a common meadow one easily sees that $M_{0\cdot z}$ is an $M_0$-module, and since $M_0\simeq R$ it is also an $R$-module. 
            
            Additionally, note that the ring homomorphisms $f_{z,z'}:M_{z'}\rightarrow M_z$ defined by $f_{z,z'}(x)=x+z$ are  morphisms of $R$-modules since given $r\in M_0$ and $x\in M_{z'}$ one has $$f_{z,z'}(r\cdot x)=r\cdot x+z=r\cdot x +r\cdot z= r\cdot (x+z) = r\cdot f_{z,z'}(x).$$
            Then $M$  is a directed lattice of rings which is also a lattice of $R$-modules.
        \end{proof}

\subsection*{Acknowledgments}

The authors would like to thank an anonymous reviewer, whose comments greatly improved this paper.

Both authors acknowledge the support of FCT - Funda\c{c}\~ao para a Ci\^{e}ncia e Tecnologia under the project: 10.54499/UIDB/04674/2020, and the research center CIMA -- Centro de Investigação em Matemática e Aplicações. 

The second author also acknowledges the support of CMAFcIO -- Centro de Matem\'{a}tica, Aplica\c{c}\~{o}es Fundamentais e Investiga\c{c}\~{a}o Operacional under the project UIDP/04561/2020.

        \bibliographystyle{plain}
\bibliography{References}

\end{document}